\documentclass{article}
\usepackage{graphicx} 
\usepackage{amssymb,amsthm,amstext,amsmath}
\usepackage{fullpage}
\usepackage{lmodern}
\usepackage{titlesec}
\usepackage[colorlinks=false,hidelinks]{hyperref}
\usepackage{biblatex}
\bibliography{bibliography.bib}

\titleformat{\section}
  {\normalfont\large}      
  {\thesection.}           
  {1em}                    
  {\centering}             

\titlespacing*{\subsection}
  {0pt}       
  {1.2ex}     
  {0.8ex}     

\newtheorem{TT}{Theorem}[section]
\newtheorem{LL}[TT]{Lemma}

\newtheorem{PP}[TT]{Proposition}
\newtheorem{DD}[TT]{Definition}
\newtheorem{EE}[TT]{Example}
\newtheorem{RR}[TT]{Remark}

\title{Interplay between Quantitative Aspects of Locally Conformally
Symplectic Geometry and Contact Dynamics}
\author{Pacôme Van Overschelde}
\date{}

\nocite{Belgun_2019}
\nocite{CandelConlon2000}

\begin{document}

\maketitle

\begin{abstract}
We investigate quantitative properties of exact locally conformally symplectic (LCS) manifolds, namely the homotheties of the Lee form that still produce an exact LCS form.
This gives the notion of elasticity of an exact LCS pair. 
Using this, we characterize LCS manifolds of the first kind.
We then generalize a result of Bazzoni and Marrero on the latter, by showing that an exact LCS manifold of rank $1$ admitting an exact LCS pair, whose complementary of elasticity is bounded, is isomorphic to an LCS mapping torus.
Conversely, we show that any LCS mapping tori over a closed contact manifold satisfies this condition, thereby providing a characterization of LCS mapping tori over closed contact manifolds.
In doing so, we establish a link between the limit values of the elasticity of an LCS mapping torus over a closed contact manifold and the Birkhoff average of the conformal factor of its contactomorphism.
As a consequence, we obtain a lower bound on the maxima and an upper bound on the minima of the various conformal factors associated with a contactomorphism on a closed contact manifold.
\end{abstract}

\tableofcontents

\section{Introduction}
A nondegenerate $2$-form $\omega$ on a $2n$-dimensional manifold $M$ is said to be \textit{locally conformally symplectic} (LCS) if there exists a closed $1$-form $\eta$ on $M$ such that $d\omega = \eta \wedge \omega$.
The form $\eta$ is then called the Lee form associated with $\omega$ and is uniquely determined.
Two LCS forms $\omega$ and $\omega'$ on $M$ are said to be conformally equivalent if there exists a function $f \in C^\infty(M)$ such that $\omega' = e^f \omega$ and an LCS structure $[\omega]$ on $M$ is an equivalence class for this relation.
\smallskip

Locally conformal symplectic structures were first studied by Lee \cite{Lee1943}, Libermann \cite{Libermann1955}, and Lichnerowicz \cite{Lichnerowicz1983} before being formalized by Vaisman \cite{Vaisman1985}. They were later studied by Banyaga \cite{7a23308de98840c3b920b67d847ed2f6}, among many others.
\smallskip

The historically most studied class of LCS structures are LCS structures of the first kind.
For these, there exists a vector field $X \in \mathfrak{X}(M)$ such that $\mathcal{L}_X \omega = 0$ and $\eta(X)$ is nonvanishing.
This is equivalent to the existence of a $1$-form $\lambda$ on $M$ such that $\omega = d\lambda - \eta \wedge \lambda$ and $d\lambda$ is of rank $2n-2$.
When the first condition is satisfied, the LCS structure is said to be exact.
LCS structures of the first kind are a particular case of exact LCS structures.
\smallskip

In this paper, we will mainly focus on exact LCS structures.
We investigate homotheties of the Lee form that still produce an exact LCS form, i.e. for which constants $c \in \mathbb{R}$ the $2$-form $d\lambda - c\, \eta \wedge \lambda$ remains nondegenerate.
We define the elasticity $E(\lambda,\eta)$ as the set of such constants, and we show that 
\bigskip
 
\noindent \textbf{Proposition 2.23.} \textit{The pair $(\lambda,\eta)$ defines an LCS structure of the first kind if and only if $E(\lambda,\eta) = \mathbb{R}_0$}.
\bigskip

Let $(N,\xi)$ be a contact manifold.
A contactomorphism $\psi \in \mathrm{Cont}(N,\xi)$ is said to be \textit{strict} if there exists a contact form $\alpha$ associated with $\xi$ such that $\psi^* \alpha = \alpha$.  
Banyaga constructed in \cite{Banyaga2002}, LCS structure of the first kind on mapping tori of a strict contactomorphism and showed conversely \cite[Theorem~2]{Banyaga2002} that a compact manifold endowed with an LCS structure of the first kind fibres over the circle, and that the fibre inherits a contact structure.
\smallskip

Subsequently, Bazzoni and Marrero showed \cite[Theorem~4.7]{Bazzoni_2018}, that under a compactness asumption on the canonical foliation induced by the Lee form, that a compact LCS manifold of the first kind is LCS-isomorphic to a mapping torus of a strict contactomorphism.
\smallskip

Exact LCS structures on mapping tori of non necessarily strict contactomorphisms were independently constructed by Chantraine and Sackel in \cite{chantraine2024productslocallyconformalsymplectic}, and by the authors in \cite{VanOverschelde2025}.
This construction involves a constant term corresponding to the size of the Lee form.
We define the set of admissible values $\mathcal{A}_{(\alpha,\psi)}$ of the LCS mapping torus associated to $\psi$ as the set of values that this constant can take.
\smallskip

We generalise the result of Bazzoni and Marrero by showing that, under certain conditions on the canonical foliation induced by the Lee form, if one weakens the first kind condition by requiring that $\eta \wedge \lambda \wedge d\lambda^{\,n-1}$ is nonvanishing, then the exact LCS manifold is isomorphic to an LCS mapping torus for a non necessarily strict contactomorphism.
\bigskip
 
\noindent \textbf{Theorem 3.3.} \textit{Let $(M,\lambda,\eta)$ be an exact LCS manifold of rank $1$ such that 
$\eta \wedge \lambda \wedge d\lambda^{\,n-1}$ is non vanishing and $Z_\lambda$ the $\eta$-Liouville vector field is complete.  
Then $(M,\lambda,\eta)$ is isomorphic to an LCS mapping torus.}
\bigskip

In the compact case, the previous condition is equivalent to elasticity $E(\lambda,\eta)$ having a bounded complementary set.
We then show that over a closed contact manifold any LCS mapping torus satisfies this condition, thereby providing a characterization of LCS mapping tori over closed contact manifolds.
\bigskip
 
\noindent \textbf{Theorem 3.10.} \textit{A closed exact LCS manifold $(M,\lambda,\eta)$ of rank~$1$ is isomorphic to the LCS mapping torus of a closed contact manifold if and only if  $E(\lambda,\eta)^{c}$ is bounded.}
\bigskip

In doing so, we establish a link between the admissible values $\mathcal{A}_{(\alpha,\psi)}$ and the sequence of partial Birkhoff averages $(A_n(h))_{n \in \mathbb{N}_0}$ of the conformal factor $h$ associated with $\psi$, where 
$$A_n(h):= \frac{1}{n}\sum_{i=0}^{n-1} h \circ \psi^i, \ \ \text{with $\psi^* \alpha = e^h \alpha$.} $$
We then determine the admissible values $\mathcal{A}_{(\alpha,\psi)}$ over a closed contact manifold as some limit values of the Birkhoff average of the conformal factor associated with $\psi$.
\bigskip
 
\noindent \textbf{Theorem 3.14.} \textit{Let $(N,\xi)$ be a closed contact manifold, $\alpha$ a contact form associated with $\xi$, and $\psi$ a contactomorphism satisfying $\psi^* \alpha = e^h \alpha$, with $h \in C^\infty(N)$.
Then \[
    \mathcal{A}_{(\alpha,\psi)}
    =
    \ ]-\infty,\ \lim_{n\to +\infty} \inf_{x \in N} \inf_{i \geq n} A_i(h)(x)[
    \ \cup\ 
    ]\lim_{n\to +\infty} \sup_{x \in N}\sup_{i \geq n} A_i(h)(x),\ +\infty[\, .
\]}
\bigskip

As a consequence, we obtain a lower bound on the maxima and an upper bound on the minima of the various conformal factors associated with a contactomorphism on a closed contact manifold.
\bigskip

\noindent \textbf{Theorem 3.15.} \textit{Let $N$ be a closed manifold, $\alpha$ a contact form on $N$, and $\psi$ a contactomorphism such that $\psi^* \alpha = e^h \alpha$, with $h \in C^\infty(N)$.
We have 
\[ \inf_{f \in C^\infty(N)} \max_{x \in N}(h(x)+f\circ \psi(x)-f(x)) = \lim_{n \to \infty} \max_{x \in N} A_n(h)(x)\]
and 
\[ \sup_{f \in C^\infty(N)} \min_{x \in N}(h(x)+f\circ \psi(x)-f(x)) = \lim_{n \to \infty} \min_{x \in N} A_n(h)(x).\]}

\section{Preliminary}
\subsection{Exact LCS manifolds}

Let \(M\) be a smooth manifold and \(\eta\) a closed \(1\)-form on \(M\).
For any \(\beta \in \Omega^*(M)\), we define the \textit{Lichnerowicz differential} of \(\beta\) with respect to \(\eta\) as \(d_\eta \beta = d\beta - \eta \wedge \beta\).
\smallskip

\begin{DD}
A $1$-form $\lambda \in \Omega^1(M)$ is said to be \textit{$\eta$-Liouville} if $d_\eta \lambda$ is non-degenerate and an \textit{exact LCS\ pair} $(\lambda,\eta)$ on $M$ consists of a closed $1$-form $\eta$ and an $\eta$-Liouville form $\lambda$.
\smallskip

Two exact LCS pairs $(\lambda,\eta)$ and $(\lambda',\eta')$ on $M$ are said to be \textit{exactly conformally equivalent} if there exists functions $g,h \in C^\infty(M)$ such that $\lambda' = e^g(\lambda + d_\eta h)$ and $\eta' = \eta + dg$.
An \textit{exact LCS  structure} $[(\lambda,\eta)]$ on $M$ is an equivalence class of exact LCS pairs under this relation, and an \textit{exact LCS manifold} $(M,[(\lambda,\eta)])$ is a smooth manifold together with an exact LCS structure.
\end{DD}

The integration of $\eta$ defines a group morphism
\(
\int \eta : \pi_1(M) \to (\mathbb{R}, +): [\gamma] \mapsto \int_\gamma \eta,
\)
which depends only on $[\eta] \in H^1(M,\mathbb{R})$.
Since $(\mathbb{R}, +)$ is abelian, the map $\int \eta$ factors uniquely through $H_1(M,\mathbb{Z})$, yielding the \textit{period morphism}
\(
\mathrm{Per}_{[\eta]} : H_1(M,\mathbb{Z}) \to \mathbb{R}.
\)

\begin{DD}[Belgun - Goertsches - Petrecca]
Let $(M,[(\lambda,\eta)])$ be an exact LCS manifold. The \textit{LCS rank} of $(M,[(\lambda,\eta)])$ is defined as the rank of the subgroup of $(\mathbb{R}, +)$ generated by the image of $\mathrm{Per}_{[\eta]}$.
\end{DD}

\begin{LL}
\label{morphisme to lee}
Let $\overline{M}$ be a smooth manifold and $\Gamma$ a group acting properly discontinuously on $\overline{M}$.
There is a correspondence between $\mathrm{Hom}(\Gamma,\mathbb{R})$ and $H^1_{dR}(\overline{M}/\Gamma)$.
\end{LL}

\begin{proof}
Equip $M := \overline{M}/\Gamma$ with the unique smooth structure for which the quotient projection
$\pi : \overline{M} \to M$ is a smooth normal covering.
Then $\Gamma \cong \pi_1(M)/\pi_\#(\pi_1(\overline{M}))$.
For every $\tau \in \mathrm{Hom}(\Gamma,\mathbb{R})$, one may extend $\tau$ to a morphism from $\pi_1(M)$ to $\mathbb{R}$ and factor it through $H_1(M,\mathbb{Z})$ to obtain an element
$\tilde{\tau} \in \mathrm{Hom}(H_1(M;\mathbb{Z}),\mathbb{R})$.
By the Universal coefficient theorem 
$\mathrm{Hom}(H_1(M;\mathbb{Z}),\mathbb{R}) \cong H^1(M;\mathbb{R})$, and by the de Rham Theorem
$H^1(M;\mathbb{R}) \cong H^1_{dR}(M)$.
Hence there exists $[\eta] \in H^1_{dR}(M)$ such that $\tilde{\tau}$ and $\mathrm{Per}_{[\eta]}$ coincide.
\end{proof}

\begin{PP}[Belgun - Goertsches - Petrecca]
\label{2.6}
Let $(\overline{M}, d\overline{\lambda})$ be an exact symplectic manifold, $\Gamma$ a group acting properly discontinuously on $\overline{M}$, and  $\tau \in \mathrm{Hom}(\Gamma,\mathbb{R})$ such that 
\(\gamma^* \overline{\lambda} = e^{\tau(\gamma)}\, \overline{\lambda}\), for every \(\gamma \in \Gamma.\)
Then one can naturally endow the quotient $\overline{M}/\Gamma$ with an exact LCS structure.
\end{PP}

\begin{proof}
By assumption, there exists a constant $c_\gamma \in \mathbb{R}$ such that \(\gamma^* \overline{\lambda} = e^{c_\gamma}\, \overline{\lambda}\), for every $\gamma \in \Gamma$.
We then define an additive group morphism 
\(\tau : \Gamma \to \mathbb{R}: \gamma \mapsto -c_\gamma\).
\smallskip

Endow $M := \overline{M}/\Gamma$ with the unique smooth structure for which the quotient projection 
$\pi : \overline{M} \to M$ is a smooth normal covering, so that \(\Gamma \cong \pi_1(M)/\pi_\#(\pi_1(\overline{M}))
\).
By Lemma~\ref{morphisme to lee}, we can associate to $\tau$ an element $\tilde{\tau} \in \mathrm{Hom}(H_1(M;\mathbb{Z}),\mathbb{R})$, and there exists 
$[\eta] \in H^1_{dR}(M)$ such that $\tilde{\tau}$ and $\mathrm{Per}_{[\eta]}$ coincide.
\smallskip

Given that $\int \pi^* \eta$ vanishes on $\pi_1(\overline{M})$, the form $\pi^* \eta$ is exact on $\overline{M}$.
Let $\mu \in C^\infty(\overline{M})$ be a primitive of $\pi^* \eta$.
Fix ${x} \in \overline{M}$ and, for every $\gamma\in \Gamma$, choose a path 
$u_{\gamma} : [0,1] \to \overline{M}$ from ${x}$ to $\gamma({x})$.
Then
\[\tau(\gamma) = \int_{\pi \circ u_{\gamma}} \eta = \int_{u_\gamma} d\mu = \mu\circ \gamma({x}) - \mu({x}).
\]
Since this equality holds for every ${x} \in \overline{M}$, we obtain
$\mu \circ \gamma - \mu = \tau(\gamma)$, for all $\gamma \in \Gamma$.
It follows that $d\mu$ and $e^{\mu}\,\overline{\lambda}$ are $\Gamma$-invariant, and so $\pi^* \eta = d\mu$ and $\pi^* \lambda = e^{\mu}\overline{\lambda}$, which define, respectively, a closed $1$-form $\eta$ and a $1$-form $\lambda$ on $M$.
As \begin{align*}
    \pi^* d_\eta \lambda&= d_{\pi^* \eta} \pi^* \lambda \\
    &=d_{d\mu}e^{\mu}\overline{\lambda} \\
    &=e^{\mu} d\overline{\lambda},
\end{align*} and $\pi$ is a local diffeomorphism, the nondegeneracy of $d\overline{\lambda}$ then implies that of $d_\eta \lambda$.
So $\lambda$ is $\eta$-Liouville, and $(\lambda,\eta)$ induces an exact LCS structure on $M$.
\end{proof}

\begin{LL}
\label{chapeaux pointu}
Let $(\overline{M}, d\overline{\lambda})$ be an exact symplectic manifold, $\Gamma$ a group acting properly discontinuously on $\overline{M}$, and $\tau \in \mathrm{Hom}(\Gamma,\mathbb{R})$ such that $\gamma^* \overline\lambda = e^{\tau(\gamma)} \overline\lambda$, for every $\gamma \in \Gamma$.
If $\mu' \in C^\infty(\overline{M}/\Gamma)$ is such that $\gamma^* \mu' = \mu' + \tau(\gamma)$, for every $\gamma \in \Gamma$, then $d\mu'$ induces a closed $1$-form $\eta'$ on $\overline{M}/\Gamma$, and $\tau$ corresponds to $[\eta']$ via Lemma~\ref{morphisme to lee}.
\end{LL}

\begin{proof}
From the previous proposition, we associate to $\tau$ an element 
$\tilde{\tau} \in \mathrm{Hom}(H_1(M;\mathbb{Z}),\mathbb{R})$, and there exists 
$[\eta] \in H^1_{dR}(M)$ such that $\tilde{\tau}$ and $\mathrm{Per}_{[\eta]}$ coincide.
\smallskip

Fix $x \in M$ and $\overline{x} \in \pi^{-1}(\{x\})$, and consider $\tilde{\gamma}$ a loop in $M$ based at $x$ corresponding to $\gamma$, i.e.\ $\tilde{\gamma}$ lifts to a path $\overline{\gamma}$ in $\overline{M}$ from $\overline{x}$ to $\gamma(\overline{x})$.
We then have
\[
\mathrm{Per}_{[\eta']}([\tilde{\gamma}]) 
= \int_{\tilde{\gamma}} \eta'
= \int_{\overline{\gamma}} d\mu'
= \mu'\circ \gamma(\overline{x}) - \mu'(\overline{x})
= \tau(\gamma).
\]
Thus $\tilde{\tau}$ and $\mathrm{Per}_{[\eta']}$ coincide, and therefore $[\eta] = [\eta']$.
\end{proof}

\begin{DD}
\label{construction tore d'application l.c.s.}
Let $(N,\xi)$ be a contact manifold, $\alpha$ be a contact form associated with $\xi$, and $\psi$ be a contactomorphism satisfying $\psi^* \alpha = e^h \alpha$, with $h \in C^\infty(N)$.
For every $k \in \mathbb{R}$, we define a diffeomorphism
\[
\rho_{(\psi,k-h)} : N \times \mathbb{R} \to N \times \mathbb{R}: (x,t) \mapsto (\psi(x), t + k - h(x)).
\]
And define the \textit{set of admissible values of $(\alpha,\psi)$} as
\[
\mathcal{A}_{(\alpha,\psi)} := \{\, k \in \mathbb{R} \mid \text{the $\mathbb{Z}$-action generated by $\rho_{(\psi,k-h)}$ is properly discontinuous} \,\}.
\]
For every $k \in \mathcal{A}_{(\alpha,\psi)}$, we define $N_{(\alpha,\psi,k)}$ as the quotient of $N \times \mathbb{R}$ by the $\mathbb{Z}$-action generated by $\rho_{(\psi,k-h)}$.
Let $(N \times \mathbb{R}, d(e^t p_1^* \alpha))$\footnote{Here, $p_1 : N \times \mathbb{R} \to N$ denotes the projection onto the first factor.
} be the ordinary symplectization of $(N,\alpha)$; one then checks that
\begin{align*}
    \rho_{(\psi,k-h)}^* e^t p_1^* \alpha 
    = e^k e^t p_1^* \alpha.
\end{align*}
If $\pi : N \times \mathbb{R} \to N_{(\alpha,\psi,k)}$ is the usual projection onto the quotient, then by the proof of Proposition~\ref{2.6}, there exists a function $\mu \in C^\infty(N \times \mathbb{R})$ satisfying $\rho_{(\psi,k-h)}^* \mu = \mu - k$ and such that the pair $(\lambda,\eta)$ defined by
\[
\pi^* \lambda = e^{t+\mu} p_1^* \alpha \ \ \text{and} \ \  
\pi^* \eta = d\mu
\]
is an exact LCS pair on $N_{(\alpha,\psi,k)}$.
The exact LCS manifold $(N_{(\alpha,\psi,k)},[(\lambda,\eta)])$ thus obtained is called the \textit{LCS mapping torus of size $k$ associated with $(\alpha,\psi)$}.
It is an exact LCS manifold of rank~$1$.
\end{DD}

\begin{EE}[Banyaga]
\label{labelo}
Let $(N,\xi)$ be a contact manifold.
A contactomorphism $\psi \in \mathrm{Cont}(N,\xi)$ is said to be \textit{strict} if there exists a contact form $\alpha$ associated with $\xi$ such that $\psi^* \alpha = \alpha$.
Then for every $k \in \mathbb{R}_0$, the diffeomorphism $\rho_{(\psi,k)}$ generates a properly discontinuous $\mathbb{Z}$-action.
And since $\rho_{(\psi,k)}^*(-t) = -t - k$, it follows from Lemma~\ref{chapeaux pointu} that \(\pi^* \lambda := p_1^* \alpha\) and \(\pi^* \eta = - dt\) define an exact LCS pair $(\lambda,\eta)$ inducing the LCS structure on $N_{(\alpha,\psi,k)}$.
\end{EE}

\begin{PP}[Belgun - Goertsches - Petrecca]
\label{lcs -> symplectic}
Let $(\lambda,\eta)$ be an exact LCS pair on a manifold $M$, and 
$\pi : \overline{M} \to M$ a normal covering such that $\pi^*\eta$ is exact.
For any choice of a primitive $\mu \in C^\infty(\overline{M})$ of $\pi^* \eta$, we have that 
\( \overline{\lambda} := e^{-\mu}\, \pi^* \lambda \) is a Liouville form on $\overline{M}$ which is preserved up to a positive factor by the Deck transformations of $\pi$.
\end{PP}

\begin{proof}
Let $\mu \in C^\infty(\overline{M})$ be a primitive of $\pi^*\eta$. For every 
$\gamma \in \mathrm{Deck}(\pi)$, the function $\mu \circ \gamma$ is again a primitive of 
$\pi^*\eta$ and $\mu - \mu \circ \gamma = c_\gamma$, where $c_\gamma$ is a constant
independent of the choice of the primitive $\mu$ of $\pi^*\eta$.
One then checks that for all $\gamma, \gamma' \in \mathrm{Deck}(\pi)$, we have
\begin{align*}
     c_{\gamma \circ \gamma'} &= \mu - \mu \circ \gamma \circ \gamma'\\
     &= \mu - \mu \circ \gamma + \mu \circ \gamma  - \mu \circ \gamma \circ \gamma'\\
     &= c_\gamma + c_{\gamma'}.
\end{align*}
Thus \(\tau : \mathrm{Deck}(\pi) \to \mathbb{R}: \gamma \mapsto -c_\gamma,\)
defines a group morphism satisfying $\mu \circ \gamma - \mu = \tau(\gamma)$ for all 
$\gamma \in \mathrm{Deck}(\pi)$.
We set \( \overline{\lambda} := e^{-\mu} \pi^* \lambda \).
As $d\overline{\lambda} = e^{-\mu} \pi^* d_\eta \lambda$ and $\pi$ is a local diffeomorphism, the nondegeneracy of $d_\eta \lambda$ implies that of $d\overline{\lambda}$.
So $\overline{\lambda}$ is a Liouville form on $\overline{M}$, and satisfies
$\gamma^* \overline{\lambda} = e^{\tau(\gamma)} \overline{\lambda}$, for every $\gamma \in \Gamma$.
\end{proof}

Let $\lambda$ be an $\eta$-Liouville form on $M$.
The \textit{$\eta$-Liouville vector field} $Z_\lambda \in \mathfrak{X}(M)$ associated with $\lambda$ is defined by \(
\iota_{Z_\lambda} d_\eta \lambda = \lambda\).
For every function $h \in C^\infty(M)$, the LCS Hamiltonian vector field $X_h$ associated with $h$ is defined by \(\iota_{X_h} d_\eta \lambda = d_\eta h\).
For all $g,h \in C^\infty(M)$, the $(\eta+dg)$-Liouville vector field associated with 
$e^g(\lambda + d_\eta h)$ is equal to $Z_\lambda + X_h$.
Indeed,
\begin{align*}
    \iota_{Z_\lambda +X_h} d_{\eta+dg}(e^g(\lambda+ d_\eta h)) &=e^g \iota_{Z_\lambda +X_h} d_{\eta}(\lambda+ d_\eta h) \\
    &= e^g (\iota_{Z_\lambda} d_\eta \lambda + \iota_{X_h} d_\eta \lambda) \\
    &= e^g(\lambda+ d_\eta h).
\end{align*}

\begin{PP}
\label{relever de liouville}
Let $(\lambda,\eta)$ be an exact LCS pair on a manifold $M$, let $\pi : \overline{M} \to M$ be a normal covering such that $\pi^*\eta$ is exact, and $\mu \in C^\infty(\overline{M})$ a primitive of $\pi^*\eta$.
The $\eta$-Liouville vector field $Z_\lambda$ lifts to $\overline{M}$ as the Liouville vector field 
$\overline{Z}$ of $\overline{\lambda} := e^{-\mu} \pi^* \lambda$.
For every $h \in C^\infty(M)$, the LCS Hamiltonian vector field $X_h \in \mathfrak{X}(M)$ lifts to a Hamiltonian vector field 
$X_{e^{-\mu} h \circ \pi} \in \mathfrak{X}(\overline{M})$ on $(\overline{M}, d\overline{\lambda})$.
\smallskip

Conversely, for every $\overline{h} \in C^\infty(\overline{M})$, the Hamiltonian vector field 
$X_{\overline{h}}$ on $(\overline{M}, d\overline{\lambda})$ descends to a LCS Hamiltonian vector field if and only if there exists a constant $c \in \mathbb{R}$ such that the function $e^\mu (\overline{h} - c)$ is $\mathrm{Deck}(\pi)$-invariant.
\end{PP}

\begin{proof}
Define $\overline{Z}, \overline{X}_h \in \mathfrak{X}(\overline{M})$ by $\pi_* \overline{Z} = Z_\lambda$ and $\pi_* \overline{X}_h = X_h$.
We then check that \begin{align*}
    \iota_{\overline{Z}} d \overline{\lambda}
&= e^{-\mu} \iota_{\pi_* Z_\lambda} d_{d\mu} \pi^* \lambda \\
    &= e^{-\mu} \pi^* (\iota_{Z_\lambda} d_\eta \lambda) \\
    &= e^{-\mu} \pi^* \lambda \\
    &= \overline{\lambda},
\end{align*} and  \begin{align*}
    \iota_{\overline{X}_h} d \overline{\lambda} &= \iota_{\overline{X}_h} e^{-\mu}\pi^* d_\eta \lambda \\
    &= e^{-\mu} \pi^*\iota_{X_h} d_\eta \lambda \\
    &= e^{-\mu} \pi^*d_{\eta} h \\
    &= d(e^{-\mu} h \circ \pi).
\end{align*}
Therefore $\overline{Z}$ coincides with the Liouville vector field 
$\overline{Z}$ of $\overline{\lambda}$ and $\overline{X}_h$ coincides with the Hamiltonian vector field 
$X_{e^{-\mu} h \circ \pi}$ associated with the function 
$e^{-\mu} h \circ \pi$.
\smallskip

If we are given $\overline{h} \in C^\infty(\overline{M})$ such that 
$e^\mu(\overline{h} - c)$ is $\mathrm{Deck}(\pi)$-invariant, then there exists 
$h \in C^\infty(M)$ such that $h \circ \pi = e^\mu(\overline{h} - c)$.
By the above, the LCS Hamiltonian vector field $X_h$ associated with $h$ lifts to 
$\overline{M}$ as the Hamiltonian vector field $X_{\overline{h}}$ associated with $\overline{h}$.
\smallskip

Conversely, if a Hamiltonian vector field $X_{\overline{h}}$ associated with 
$\overline{h} \in C^\infty(\overline{M})$ descends to an LCS Hamiltonian vector field $X_h$
for a function $h \in C^\infty(M)$, then \begin{align*}
    d \overline{h} &= \iota_{X_{\overline{h}}} d \overline{\lambda} \\
    &= \iota_{X_{\overline{h}}} e^{-\mu} \pi^* d_\eta \lambda \\
    &= e^{-\mu} \pi^* \iota_{X_h} d_\eta \lambda \\
    &= e^{-\mu} \pi^* d_\eta h \\
    &= d(e^{-\mu} h \circ \pi).
\end{align*}
Hence $\overline{h} = e^{-\mu} h \circ \pi + c$ for a constant $c \in \mathbb{R}$, so that 
$e^\mu(\overline{h} - c)$ is $\mathrm{Deck}(\pi)$-invariant.
\end{proof}

\begin{DD}
An \textit{exact conformal symplectomorphism} between two exact LCS manifolds $(M,[(\lambda,\eta)])$ and $(M',[(\lambda',\eta')])$ is a diffeomorphism 
$\varphi \in \mathrm{Diff}(M,M')$ such that the exact LCS pairs 
$(\varphi^* \lambda', \varphi^* \eta')$ and $(\lambda,\eta)$ are exactly conformally equivalent, i.e.\ there exists functions $g,h \in C^\infty(M)$ such that
\[
\varphi^* \lambda' = e^g(\lambda + d_\eta h)
\ \ \text{and} \ \ 
\varphi^* \eta' = \eta + dg.
\]
\end{DD}

\begin{EE}
\label{2.15}
Let $(M,[(\lambda,\eta)])$ be an exact LCS manifold and $X_h$ the LCS Hamiltonian vector field associated with $h \in C^\infty(M)$.
By Cartan’s formula, we have $\mathcal{L}_{X_h} \eta = d\,\eta(X_h)$.  
This implies that
\[ \frac{d}{dt}[\varphi_{X_h}^t{}^* \eta]|_{t=s} = d \eta(X_h)\circ \varphi_{X_h}^s, \ \  \text{with} \ \  \varphi^0_{X_h}{}^* \eta = \eta.\]
This equation has the unique solution
\[\varphi^t_{X_h}{}^* \eta = \eta + d\Big(\int_0^t \eta(X_h) \circ \varphi_{X_h}^s ds \Big).\]
Set $g_t := \int_0^t \eta(X_h) \circ \varphi_{X_h}^s\, ds$.  
Applying Cartan’s formula once more, we obtain
\begin{align*}
    \mathcal{L}_{X_h} \lambda &= \iota_{X_h} d \lambda + d \circ \iota_{X_h} \lambda \\
    &= \iota_{X_h} (d_\eta \lambda + \eta \wedge \lambda) + d_\eta (\iota_{X_h} \lambda) + \eta \wedge \iota_{X_h} \lambda \\
    &= \eta(X_h) \lambda + d_\eta (\iota_{X_h} \lambda - h). 
\end{align*}
This implies that
\[ \frac{d}{dt}[\varphi_{X_h}^t{}^* \lambda]|_{t=s} = \eta(X_h)\circ \varphi_{X_h}^s \  \varphi_{X_h}^s{}^*\lambda + \varphi_{X_h}^s{}^* d_\eta( \iota_{X_h} \lambda + h).\]
Now,
\begin{align*}
    \varphi_{X_h}^s{}^* d_\eta( \iota_{X_h} \lambda + h) &= d_{\varphi^s_{X_h}{}^* \eta} (\iota_{X_h} \varphi_{X_h}^s{}^* \lambda+h)\\
    &= d_{\eta + dg_s}(\iota_{X_h} \varphi_{X_h}^s{}^* \lambda+h) \\
    &= e^{g_s} d_\eta(e^{-g_s} (\iota_{X_h} \varphi_{X_h}^s{}^* \lambda+h)).
\end{align*}
Thus $\varphi_{X_h}^s{}^* \lambda$ solves
\[\frac{d}{dt}[\varphi_{X_h}^t{}^* \lambda]|_{t=s} = \eta(X_h)\circ \varphi_{X_h}^s \  \varphi_{X_h}^s{}^*\lambda + e^{g_s} d_\eta(e^{-g_s} (\iota_{X_h} \varphi_{X_h}^s{}^* \lambda+h)),\]
with the initial condition $\varphi_{X_h}^0{}^* \lambda = \lambda$.  
This implies that
\[
    \varphi_{X_h}^t{}^* \lambda
    = e^{g_t}\Big( \lambda
        + d_\eta \Big(
            \int_0^t e^{-g_s} (\iota_{X_h} \lambda + h)\circ \varphi_{X_h}^s\, ds
        \Big)
      \Big).
\]
Hence the time-$t$ flow of a Hamiltonian vector field is an exact conformal symplectomorphism.
\end{EE}

\begin{PP}
\label{2.16}
Let $(\overline{M}, d\overline{\lambda})$ and $(\overline{M}', d\overline{\lambda}')$ be two exact symplectic manifolds, $\Gamma$ and $\Gamma'$ be groups acting properly discontinuously on $\overline{M}$ and $\overline{M}'$ respectively, and $\tau \in \mathrm{Hom}(\Gamma,\mathbb{R})$ and $\tau' \in \mathrm{Hom}(\Gamma',\mathbb{R})$ be such that $\gamma^* \overline{\lambda} = e^{\tau(\gamma)} \overline{\lambda}$ and $\gamma'^* \overline{\lambda}' = e^{\tau'(\gamma')} \overline{\lambda}'$ for all $\gamma \in \Gamma$ and $\gamma' \in \Gamma'$.
\smallskip

If $\overline{\varphi} \in \mathrm{Diff}(\overline{M},\overline{M}')$ permutes the actions of $\Gamma$ and $\Gamma'$, satisfies 
$\overline{\varphi}^* \overline{\lambda}' = \overline{\lambda} + d\overline{h}$ with 
$\overline{h} \in C^\infty(\overline{M})$, and if there exists a constant $c \in \mathbb{R}$ such that 
$e^\mu (\overline{h} - c)$ is $\Gamma$-invariant, then $\overline{\varphi}$ induces an exact conformal symplectomorphism
\(\varphi : (\overline{M}/\Gamma,[(\lambda,\eta)]) \to (\overline{M}'/\Gamma',[(\lambda',\eta')])\)
between the exact LCS manifolds induced on the quotient by Proposition~\ref{2.6}.
\end{PP}

\begin{proof}
Endow $M := \overline{M}/\Gamma$ and $M' := \overline{M}'/\Gamma'$ with the unique differentiable structures for which the quotient projections $\pi : \overline{M} \to M$ and $\pi' : \overline{M}' \to M'$ are respectively normal coverings.
\smallskip

From Proposition \ref{2.6}, choose $\mu\in C^\infty(\overline{M})$ and $\mu'\in C^\infty(\overline{M}')$ satisfying $\gamma^* \mu = \mu + \tau(\gamma)$ and $\gamma'^* \mu' = \mu' + \tau'(\gamma')$, for all $\gamma \in \Gamma$ and $\gamma' \in \Gamma'$.
We construct exact LCS pairs $(\lambda,\eta)$ and $(\lambda',\eta')$ on $M$ and $M'$ defined, respectively, by $\pi^* \eta = d \mu$, $\pi^* \lambda = e^{\mu} \overline{\lambda}$ and $\pi'^* \eta' = d \mu'$, $\pi'^* \lambda' = e^{\mu'} \overline{\lambda}'$.
\smallskip

Since $\overline{\varphi}$ permutes the actions of $\Gamma$ and $\Gamma'$, there exists $\varphi\in \text{Diff}(M,M')$ such that $\varphi \circ \pi = \pi' \circ \overline{\varphi}$ and $\varphi_\#: \pi_1(M) \to \pi_1(M')$ induces an isomorphism between $\Gamma$ and $\Gamma'$ satisfying $\varphi_\#(\gamma) \circ \overline{\varphi} = \overline{\varphi} \circ \gamma$, for all $\gamma \in \Gamma$.
If $e^\mu (\overline{h}-c)$ is $\Gamma$-invariant, there exists $h \in C^\infty(M)$ such that $ e^\mu (\overline{h}-c) = h \circ \pi$ and $\gamma^* d\overline{h} = e^{-\tau(\gamma)} d \overline{h}$, for all $\gamma \in \Gamma$.
Hence
\begin{align*}
    e^{-\tau'(\varphi_\#(\gamma))}\overline{\varphi}^*  \overline{\lambda}' &= \overline{\varphi}^* \circ \varphi_\#(\gamma)^* \overline{\lambda}' \\
    &= \gamma^* \circ \overline{\varphi}^* \overline{\lambda}' \\
    &= \gamma^* (\overline{\lambda} + d \overline{h}) \\
    &= e^{-\tau(\gamma)} (\overline{\lambda} + d \overline{h}),
\end{align*}
implies that $\tau = \tau' \circ \varphi_\#$.
Thus $\mu' \circ \overline{\varphi} - \mu$ is $\Gamma$-invariant and there exists $g \in C^\infty(M)$ such that $\mu' \circ \overline{\varphi} - \mu = g \circ \pi$. Indeed, for all $\gamma \in \Gamma$, we have
\begin{align*}
\gamma^* (\mu' \circ \overline{\varphi} - \mu) &= \mu' \circ \varphi_\#(\gamma) \circ \overline{\varphi} - \mu \circ \gamma \\
&= (\mu' + \tau'(\varphi_\#(\gamma))) \circ \overline{\varphi} - (\mu + \tau(\gamma))\\
&= \mu' \circ \overline{\varphi} - \mu + \tau'(\varphi_\#(\gamma)) - \tau(\gamma) \\
&= \mu' \circ \overline{\varphi} - \mu.
\end{align*}
Then,
\begin{align*}
    \pi^* \circ \varphi^* \lambda' &= \overline{\varphi}^* \circ \pi'^* \lambda' \\
    &=\overline{\varphi}^* e^{\mu'} \overline{\lambda}' \\
    &= e^{\mu' \circ \overline{\varphi}} (\overline{\lambda}+d\overline{h}) \\
    &= e^{\mu' \circ \overline{\varphi}-\mu} ( \pi^* \lambda + \pi^*d_{d\mu} e^\mu (\overline{h}-c)) \\
    &= \pi^* e^g (\lambda+d_\eta h).
\end{align*}
And since $\pi$ is a local diffeomorphism, $\varphi$ is an exact conformal symplectomorphism between $(M,[(\lambda,\eta)])$ and $(M',[(\lambda',\eta')])$.
\end{proof}

\begin{LL}
\label{castafior}
The set of admissible values and the isomorphism class of LCS mapping tori do not depend on the choice of contact form.
\end{LL}

\begin{proof}
Let $(N,\xi)$ be a contact manifold, $\alpha$ a contact form associated with $\xi$, and $\psi \in \text{Cont}(N,\xi)$ satisfying $\psi^* \alpha = e^h \alpha$ with $h \in C^\infty(N)$.
For any function $f \in C^\infty(N)$, we have $\psi^*(e^f \alpha) = e^{h+f\circ \psi-f}(e^f\alpha)$.
We define an exact symplectomorphism
$$\overline{\varphi} : (N \times \mathbb{R},d(e^tp_1^* \alpha)) \to (N \times \mathbb{R},d(e^t p_1^*(e^f\alpha))) : (x,t) \mapsto (x,t-f(x))$$
and we verify that for $(x,t) \in N \times \mathbb{R}$ and $k \in \mathbb{R}$, we have
\begin{align*}
    \overline{\varphi} \circ \rho_{(\psi,k-h)}(x,t) &= \overline{\varphi}(\psi(x),t+k-h(x)) \\
    &= (\psi(x),t+k-h(x)-f\circ \psi(x)) \\
    &= \rho_{(\psi, k-h-f\circ \psi+f)}(x,t-f(x)) \\
    &= \rho_{(\psi, k-h-f\circ \psi+f)} \circ \overline{\varphi}(x,t).
\end{align*}
The action generated by $\rho_{(\psi,k-h-f\circ \psi+f)}$ is properly discontinuous if and only if the one generated by $\rho_{(\psi,k-h)}$ is as well, and therefore $\mathcal{A}_{(\alpha,\psi)} = \mathcal{A}_{(e^f\alpha,\psi)}$.
\smallskip

Moreover, since $\overline{\varphi}^* (e^t p_1^*(e^f\alpha)) = e^t p_1^* \alpha$, by Proposition \ref{2.16}, $\overline{\varphi}$ induces an exact conformal symplectomorphism $\varphi$ between the mapping tori $(N_{(\alpha,\psi,k)},[(\lambda,\eta)])$ and $(N_{(e^f\alpha,\psi,k)},[(\lambda',\eta')])$.
\end{proof}

\begin{PP}
\label{2.28}
Let $(M,[(\lambda,\eta)])$ and $(M',[(\lambda',\eta')])$ be two exact LCS manifolds, and $\pi : \overline{M} \to M$ and $\pi' : \overline{M}' \to M'$ two normal coverings such that $\pi^* \eta$ and $\pi'^* \eta'$ are exact.
If $\varphi$ is an exact conformal symplectomorphism between $(M,[(\lambda,\eta)])$ and $(M',[(\lambda',\eta')])$ which lifts to a diffeomorphism $\overline{\varphi}: \overline{M} \to \overline{M}'$, then for any choice of a primitive $\mu'$ of $\pi'^* \eta'$, there exists a primitive $\mu$ of $\pi^* \eta$ such that $\overline{\varphi}$ is an exact symplectomorphism between $(\overline{M},d\overline\lambda)$ and $(\overline{M}',d\overline\lambda')$, where $\overline\lambda := e^{-\mu} \pi^* \lambda$ and $\overline\lambda' := e^{-\mu'} \pi'^* \lambda'$.
\end{PP}

\begin{proof}
According to Proposition \ref{lcs -> symplectic}, for any choice of a primitive $\mu'$ of $\pi'^* \eta'$, we have that $\overline\lambda' := e^{-\mu'} \pi'^* \lambda'$ is a Liouville form on $\overline{M}'$.
By hypothesis, $\varphi : M \to M'$ lifts to a diffeomorphism $\overline{\varphi} : \overline{M} \to \overline{M}'$ satisfying $\varphi \circ \pi = \pi' \circ \overline{\varphi}$, and there exist $g,h \in C^\infty(M)$ such that $\varphi^* \lambda' = e^g(\lambda + d_\eta h)$ and $\varphi^* \eta' = \eta + dg$.
We then verify that \begin{align*}
    \pi^* \eta &= \pi^* \circ \varphi^* \eta' - \pi^* dg \\
    &= \overline{\varphi}^* \circ \pi'^* \eta' - d(g\circ \pi) \\
    &= d( \mu' \circ \overline{\varphi} - g \circ \pi).
\end{align*}
We then set $\mu := \mu' \circ \overline{\varphi} - g \circ \pi$ and compute
\begin{align*}
\overline{\varphi}^* (e^{-\mu'} \pi'^* \lambda') & = e^{-\mu' \circ \overline{\varphi}} (\pi'\circ \overline{\varphi})^* \lambda' \\
&= e^{-\mu' \circ \overline{\varphi}} \pi^* \circ \varphi^* \lambda' \\
&= e^{g \circ \pi - \mu' \circ \overline{\varphi}} \pi^*(\lambda + d_\eta h) \\
&= e^{g \circ \pi - \mu' \circ\overline{\varphi}} \pi^*\lambda + e^{g \circ \pi - \mu' \circ \overline{\varphi}} d_{\mu' \circ \overline{\varphi} -g \circ \pi} \pi^*h) \\
&= e^{g \circ \pi - \mu' \circ \overline{\varphi}} \pi^*\lambda + d(e^{g \circ \pi - \mu' \circ \overline{\varphi}} \pi^* h).
\end{align*}
\end{proof}

\begin{DD}[Belgun - Goertsches - Petrecca]
Let $(M,[(\lambda,\eta)])$ be an exact LCS manifold and let $\pi : M_0 \to M$ be a normal covering such that $\pi^* \eta$ is exact.
We say that the covering $\pi$ is \textit{minimal} if $\text{Per}_{[\eta]}$ is injective, i.e. the rank of $\text{Deck}(\pi)$ is equal to the LCS rank of $(M,[(\lambda,\eta)])$.
\end{DD}

\begin{PP}[Belgun - Goertsches - Petrecca]
\label{2.10}
Every exact LCS manifold admits a unique minimal covering up to isomorphism.
\end{PP}

\begin{proof}
Let $(M,[(\lambda,\eta)])$ be an exact LCS manifold and $\Tilde{\pi}:\Tilde{M} \to M$ its universal covering.
There is a bijective correspondence between the normal subgroups of $\pi_1(M) \cong \text{Deck}(\Tilde{\pi})$ and the normal coverings of $M$ up to isomorphism. \smallskip

By the Poincaré lemma, since $\Tilde{M}$ is simply connected, the form $\Tilde{\pi}^* \eta$ is exact.
Let $\Tilde{\mu} \in C^\infty(\Tilde{M})$ be a primitive of $\Tilde{\pi}^* \eta$.
By Proposition \ref{lcs -> symplectic}, we have that $\Tilde{\lambda} := e^{-\Tilde{\mu}} \Tilde{\pi}^* \lambda$ is a Liouville form on $\Tilde{M}$ such that $\gamma^* \Tilde{\lambda} = e^{-\tau(\gamma)} \Tilde{\lambda}$ for all $\gamma \in \pi_1(M)$.
The morphism $\tau : \pi_1(M) \to \mathbb{R}$ depends neither on the primitive $\Tilde{\mu}$ of $\Tilde{\pi}^* \eta$ nor on the choice of representative of the conformal class.
\smallskip

Consider the normal subgroup $\ker \tau \subseteq \pi_1(M)$. Since $\Tilde{\mu} \circ \gamma - \Tilde{\mu} = \tau(\gamma)$, the function $\Tilde{\mu}$ descends to the quotient $\overline{M} := \Tilde{M}/\ker \tau$ as a smooth function $\mu \in C^\infty(\overline{M})$.
If we denote by $\pi: \overline{M} \to M$ the normal covering associated to $\ker \tau \subseteq \pi_1(M)$, then we have $\pi^* \eta = d\mu$, which is exact.
Finally, by the first isomorphism theorem, $\tau$ induces an injective morphism between $\pi_1(M)/\ker \tau$ and $\mathbb{R}$.
\end{proof}

\begin{EE}
\label{2.11}
Let $(Q,\tilde\lambda)$ be an exact symplectic manifold of dimension $2n-2$, and $(Q \times \mathbb{S}^1, p_2^*d\theta - p_1^*\tilde\lambda)$ its contactization.
For any $s \in \mathbb{R}$, we define a closed $1$-form $\eta_s$ on $Q \times \mathbb{S}^1 \times \mathbb{S}^1$ by $\eta_s := p_2^*d\theta - s p_3^*d\theta$.
Setting $\overline{\lambda} = p_2^*d\theta - p_1^*\tilde\lambda$, one verifies that
\begin{align*}
    (d_{\eta_s} \overline{\lambda})^n &= (d\overline{\lambda})^n - n \eta_s \wedge \overline{\lambda} \wedge (d\overline{\lambda})^{n-1} \\
    &= - n \eta_s \wedge \overline{\lambda} \wedge (d\overline{\lambda})^{n-1} \\
    &= -n (p_2^* d\theta - s p_3^* d\theta) \wedge (p_2^* d\theta - \tilde\lambda) \wedge d\tilde\lambda^{n-1} \\
    &= ns\, p_3^* d\theta \wedge p_2^* d\theta \wedge d\tilde\lambda^{\,n-1}.
\end{align*}
Hence $(\overline{\lambda},\eta_s)$ defines an exact LCS pair on $Q \times \mathbb{S}^1 \times \mathbb{S}^1$.
If we consider the covering
\[
\pi : Q \times \mathbb{R}^2 \to Q \times \mathbb{S}^1 \times \mathbb{S}^1 : (x,t,s) \mapsto (x,\theta(t),\theta(s))
\footnote{Where $\theta : \mathbb{R} \to \mathbb{S}^1$ is the projection to the quotient of $\mathbb{R}$ by the action of $\mathbb{Z}$ generated by $t \mapsto t+1$.}.
\]
Then $\text{Deck}(\pi) \cong \mathbb{Z}^2$, and $\pi^* \eta_s = p_2^* dt - s p_2^* dt$ is exact.
Thus, if $s$ is irrational, then $(Q \times \mathbb{S}^1 \times \mathbb{S}^1,[(\overline{\lambda},\eta_s)])$ has rank $2$ and $\pi$ is minimal.
Whereas if $s$ is rational, then $(Q \times \mathbb{S}^1 \times \mathbb{S}^1,[(\overline{\lambda},\eta_s)])$ has rank $1$ and $\pi$ is not minimal.
\end{EE}

\begin{LL}
\label{2.19}
An exact conformal symplectomorphism between two exact LCS manifolds satisfies the lifting criterion between their minimal coverings.
\end{LL}

\begin{proof}
Let $(M,[(\lambda,\eta)])$ and $(M',[(\lambda',\eta')])$ be two exact LCS manifolds, $\varphi : M \to M'$ be a exact conformal symplectomorphism satisfying $\varphi^* \eta' = \eta + dg$ for some function $g \in C^\infty(M)$, and $\pi : \overline{M} \to M$ and $\pi' : \overline{M}' \to M'$ be the minimal coverings respectively associated to $(M,[(\lambda,\eta)])$ and $(M',[(\lambda',\eta')])$.
\smallskip

Since $\varphi^* \eta' = \eta + dg$, we have
\[
\int \eta' \circ \varphi_\# = \int \varphi^* \eta' = \int \eta + dg = \int \eta.
\]
Under the assumption that $\pi$ and $\pi'$ are minimal, we have
\[
\ker \int \eta = \pi_\#(\pi_1(\overline{M})) \ \  \text{and} \ \ \ker \int \eta' = \pi'_\#(\pi_1(\overline{M}')).
\]
Therefore,
\[\varphi_\# (\pi_\# (\pi_1(\overline{M}))) \subseteq \pi'_\# (\pi_1(\overline{M}')) \ \  \text{and} \ \  \varphi^{-1}_\# (\pi'_\# (\pi_1(\overline{M}'))) \subseteq \pi_\# (\pi_1(\overline{M})).
\]
\end{proof}

\subsection{Elasticity}

We define the \textit{elasticity} of an exact LCS pair $(\lambda,\eta)$ on $M$ as \(E(\lambda,\eta) := \{c \in \mathbb{R} \mid d_{c\eta} \lambda \ \text{is nondegenerate}\}.\)
It is not an invariant of the exact LCS structure.
If $g,h \in C^\infty(M)$ and $c \in E(\lambda,\eta)$, then without any additional assumption $d_{c(\eta+dg)} e^g(\lambda+d_\eta h)$ is not necessarily nondegenerate.

\begin{RR}
\label{rk 12}
However, if $\varphi^* \lambda' = e^g(\lambda+d_\eta h)$ and $\varphi^* \eta' = \eta + dg$, then $E(\lambda',\eta') = E(e^g(\lambda+d_\eta h),\eta + dg)$.
\end{RR}

Since nondegeneracy is an open condition, the elasticity is an open subset of $\mathbb{R}$.
Because an exact symplectic manifold cannot be closed, if $M$ is closed then it is an open subset of $\mathbb{R}_0$.

\begin{LL}
\label{etaZnonzero}
Let $(\lambda,\eta)$ be an exact LCS pair on a manifold $M$ of dimension $2n$, and $Z_\lambda$ the $\eta$-Liouville vector field associated with $\lambda$.
The function $\eta(Z_\lambda)$ is nonvanishing if and only if $\eta \wedge \lambda \wedge d\lambda^{\,n-1}$ is nonvanishing.
\end{LL}

\begin{proof}
If $\lambda$ vanishes at a point, then its $\eta$-Liouville vector field $Z_\lambda$ also vanishes at that point, as do $\eta(Z_\lambda)$ and $\eta \wedge \lambda \wedge d\lambda^{n-1}$.
On the other hand, if $Z_\lambda$ is nonvanishing, then since $d_\eta \lambda$ is nondegenerate, we have
\(\lambda \wedge d\lambda^{n-1} = \frac{1}{n} \iota_{Z_\lambda} (d_\eta \lambda^n)\),
which is nonvanishing.
Moreover, since \(\iota_{Z_\lambda} (\eta \wedge \lambda \wedge d\lambda^{n-1}) = \eta(Z_\lambda)\, \lambda \wedge d\lambda^{n-1}\), we obtain that $\eta(Z_\lambda)$ is nonvanishing if and only if $\eta \wedge \lambda \wedge d\lambda^{n-1}$ is nonvanishing.
\end{proof}

\begin{PP}
\label{image élasic}
\label{4.2}
Let $(\lambda,\eta)$ be an exact LCS pair on a manifold $M$, and let $Z_\lambda$ be the $\eta$-Liouville vector field associated with $\lambda$.
Then
\[
E(\lambda,\eta) \subseteq \text{Im}\left(\frac{1+ \eta(Z_\lambda)}{\eta(Z_\lambda)}\right)^{c}.
\]
Moreover, if $\lambda$ is nonvanishing, then the above inclusion is an equality.
\end{PP}

\begin{proof}
Set
$$\mathcal{O}_{\eta(Z_\lambda)} := \{x \in M \mid \eta(Z_\lambda)_x \neq 0\} \subseteq \mathcal{O}_\lambda := \{x \in M \mid \lambda_x \neq 0\}.$$
These are two open subsets of $M$.
On $\mathcal{O}_\lambda$, the vector field $Z_\lambda$ is nonvanishing, so $\lambda \wedge d\lambda^{n-1} = \frac{1}{n}\iota_{Z_\lambda} (d_\eta \lambda)^n$ is also nonvanishing.
For any $c \in \mathbb{R}$, we then have
\begin{align*}
    \iota_{Z_\lambda} (d_{c\eta}\lambda)^n &= \iota_{Z_\lambda} ( d\lambda^n - c n \, \eta \wedge \lambda \wedge d\lambda^{n-1}) \\
    &= n (\iota_{Z_\lambda} d\lambda)\wedge d\lambda^{n-1} - cn\, \eta(Z_\lambda)\, \lambda \wedge d\lambda^{n-1} \\
    &= n (1+\eta(Z_\lambda))\, \lambda \wedge d\lambda^{n-1} - cn\, \eta(Z_\lambda)\, \lambda \wedge d\lambda^{n-1} \\
    &= n \big(1 + (1-c)\eta(Z_\lambda)\big)\, \lambda \wedge d\lambda^{n-1}.
\end{align*}
On $\mathcal{O}_{\eta(Z_\lambda)}$, the function \(1 + (1-c)\eta(Z_\lambda)\) vanishes if and only if 
\(c \in \text{Im}\Big(\frac{1+\eta(Z_\lambda)}{\eta(Z_\lambda)}\Big)\).
Hence,
\[
E(\lambda|_{\mathcal{O}_{\eta(Z_\lambda)}},\eta|_{\mathcal{O}_{\eta(Z_\lambda)}})
= \text{Im}\left(\frac{1+\eta(Z_\lambda)}{\eta(Z_\lambda)}\right)^c.
\]
Whereas on $\mathcal{O}_\lambda \setminus \mathcal{O}_{\eta(Z_\lambda)}$, we have 
\(\iota_{Z_\lambda} (d_{c\eta}\lambda)^n = n\, \lambda \wedge d\lambda^{n-1}\),
which is nonvanishing independently of $c\in \mathbb{R}$, and therefore
\(
E(\lambda|_{\mathcal{O}_\lambda \setminus \mathcal{O}_{\eta(Z_\lambda)}},\eta|_{\mathcal{O}_\lambda \setminus \mathcal{O}_{\eta(Z_\lambda)}})
= \mathbb{R}.
\)
Thus,
\[
E(\lambda,\eta) \subseteq E(\lambda|_{\mathcal{O}_\lambda},\eta|_{\mathcal{O}_\lambda}) = E(\lambda|_{\mathcal{O}_{\eta(Z_\lambda)}},\eta|_{\mathcal{O}_{\eta(Z_\lambda)}}) \cap E(\lambda|_{\mathcal{O}_\lambda \setminus \mathcal{O}_{\eta(Z_\lambda)}},\eta|_{\mathcal{O}_\lambda \setminus \mathcal{O}_{\eta(Z_\lambda)}})= \text{Im}\left(\frac{1+\eta(Z_\lambda)}{\eta(Z_\lambda)}\right)^c.\]
\end{proof}

\begin{DD}
\label{4.4}
An exact LCS pair $(\lambda,\eta)$ on a manifold $M$ of dimension $2n$ is said to be of the first kind if $d\lambda^n = 0$ and $\eta \wedge \lambda \wedge d\lambda^{n-1}$ is nonvanishing.
An exact LCS manifold $(M,[(\lambda,\eta)])$ is said to be \textit{of the first kind} if it admits a representative of the exact LCS structure of the first kind.
\end{DD}

\begin{PP}
\label{equivelence first kind}
Let $(\lambda,\eta)$ be an exact LCS pair on a manifold $M$. The pair $(\lambda,\eta)$ is of the first kind if and only if $E(\lambda,\eta) = \mathbb{R}_0$.
\end{PP}

\begin{proof}
If $d\lambda^n = 0$ and $\eta \wedge \lambda \wedge d\lambda^{n-1}$ is nonvanishing,
then for every $c \in \mathbb{R}_0$ we have
\(d_{c\eta}\lambda^n = c n\, \eta \wedge \lambda \wedge d\lambda^{n-1}\),
which is nonvanishing.
Hence $E(\lambda,\eta) = \mathbb{R}_0$.
\smallskip

Conversely, assume $E(\lambda,\eta) = \mathbb{R}_0$.
By Proposition \ref{image élasic}, we have $\eta(Z_\lambda) = -1$.
Thus $\iota_{Z_\lambda} d\lambda = (1+\eta(Z_\lambda)) \lambda = 0$, which implies that $d\lambda^n = 0$.
Moreover, since $\eta(Z_\lambda) \neq 0$, by Lemma \ref{etaZnonzero} we obtain that
\(\eta \wedge \lambda \wedge d\lambda^{n-1}\) is nonvanishing, and therefore $(\lambda,\eta)$ is of the first kind.
\end{proof}

\begin{LL}
\label{stristri}
Let $(N,\xi)$ be a contact manifold.
A contactomorphism $\psi \in \text{Cont}(N,\xi)$ is strict if and only if for every contact form $\alpha$ associated with $\xi$ satisfying $\psi^* \alpha = e^h \alpha$, with $h \in C^\infty(N)$,
there exists a function $f \in C^\infty(N)$ such that $h = f - f \circ \psi$.
\end{LL}

\begin{PP}
\label{4.6}
Let $(N,\xi)$ be a contact manifold, $\alpha$ a contact form associated with $\xi$, let $\psi \in \text{Cont}(N,\xi)$ be a contactomorphism satisfying $\psi^* \alpha = e^h \alpha$, with $h \in C^\infty(N)$, and $k \in \mathcal{A}_{(\alpha,\psi)}$.
Then $\psi$ is strict if and only if there exists an exact LCS pair $(\lambda,\eta)$ of the first kind, defined by $\pi^* \lambda = e^{t+\mu} p_1^* \alpha$ and $\pi^* \eta = d\mu$, with $\mu \in C^\infty(N \times \mathbb{R})$ satisfying $\rho_{(\psi,k-h)}^* \mu = \mu - k$.
Here, $\pi : N \times \mathbb{R} \to N_{(\alpha,\psi,k)}$ is the usual quotient projection.
\end{PP}

\begin{proof}
If $\psi$ is strict, then by Lemma \ref{stristri}, there exists a function $f \in C^\infty(N)$ such that $h = f - f \circ \psi.$
Set $\mu = f \circ p_1 - t$, and check that for every $(x,t) \in N \times \mathbb{R}$,
\begin{align*}
    \mu \circ \rho_{(\psi,k-h)}(x,t) &= \mu(\psi(x),t+k-f(x)+f\circ \psi(x))\\
    &= f\circ \psi(x) - t - k + f(x) - f\circ \psi(x) \\
    &= \mu - k.
\end{align*}
By Lemma \ref{chapeaux pointu}, the forms $\pi^* \lambda = e^{t+\mu} p_1^* \alpha$ and $\pi^* \eta = d\mu$ define an exact LCS pair $(\lambda,\eta)$ on $N_{(\alpha,\psi,k)}$.
Since $\partial_t \mu = -1$, by Propositions \ref{4.2} and \ref{equivelence first kind}, the pair $(\lambda,\eta)$ is of the first kind.
\smallskip

Conversely, suppose there exists an exact LCS pair of the first kind $(\lambda,\eta)$ defined by $\pi^* \lambda = e^{t+\mu} p_1^* \alpha$ and $\pi^* \eta = d\mu$, with $\mu \in C^\infty(N\times \mathbb{R})$ satisfying $\rho_{(\psi,k-h)}^* \mu = \mu - k$.
Then $\eta(Z_\lambda) = -1$, which implies that $\partial_t \mu = -1$.
Set $\mu_0(x) := \mu(x,0)$; then $\mu(x,t) = \mu_0(x) - t$.
Moreover, for every $(x,t) \in N \times \mathbb{R}$, we have
\begin{align*}
    \mu \circ \rho_{(\psi,k-h)}(x,t) &= \mu_0 \circ \psi(x) - t - k + h(x) \\
    &= \mu_0(x) - t - k.
\end{align*}
Therefore $h = \mu_0 - \mu_0 \circ \psi$, and
\(\psi^* (e^{\mu_0} \alpha) = e^{\mu_0 \circ \psi} e^{\mu_0 - \mu_0 \circ \psi} \alpha = e^{\mu_0} \alpha.\)
\end{proof}

\section{Main results}
\subsection{Characterization of LCS mapping tori}

In what follows, all manifolds are assumed to be connected.
If, in addition, we assume that the closed $1$-form $\eta$ is nonvanishing, then its kernel $\ker \eta$ defines an involutive distribution inducing a codimension-$1$ foliation $\mathcal{F}_\eta$ on $M$.

\begin{LL}[Bazzoni - Marrero]
\label{1.5}
Let $\eta$ be a nonvanishing closed $1$-form on a smooth manifold $M$, and $U \in \mathfrak{X}(M)$ a complete vector field such that $\eta(U) = 1$.
For every leaf $N \in \mathcal{F}_\eta$, the restriction of the flow of $U$ to $N \times \mathbb{R}$ is a surjective local diffeomorphism onto $M$.
\end{LL}

\begin{proof}
For every $x \in N$ and every $Y_x \in T_xN$, we have
\[\varphi_U{}_{*(x,0)} \partial_t{}_{(x,0)} = U_x 
\ \  \text{and} \ \ 
\varphi_U{}_{*(x,0)} Y_{(x,0)} = Y_x.
\]
Thus, for all $c,s \in \mathbb{R}$, we have
\begin{align*}
    \varphi_U{}_{*(x,s)} (Y + c \partial_t)_{(x,s)}
    &= (\varphi_U \circ \varphi_{\partial_t}^s)_{*(x,0)} (Y + c \partial_t)_{(x,0)} \\
    &= (\varphi_U^s \circ \varphi_U)_{*(x,0)} (Y + c \partial_t)_{(x,0)} \\
    &= {\varphi_U^s}_{*x} (Y_x + c U_x) \\
    &= {\varphi_U^s}_{*x} (Y_x) + c\, U_{\varphi_U^s(x)}.
\end{align*}
Since $\varphi_U^s{}_{*x}$ is a linear isomorphism, the flow of $U$ preserves the distribution $\ker \eta$, and
\(T_{\varphi_U^s(x)} M = \ker \eta_{\varphi_U^s(x)} \oplus \langle U_{\varphi_U^s(x)} \rangle\),
we conclude that $\varphi_U|_{N \times \mathbb{R}}$ is a local diffeomorphism.
\smallskip

Since $\varphi_U|_{N \times \mathbb{R}}$ is a local diffeomorphism, the set $\varphi_U(N \times \mathbb{R})$ is open in $M$.
Assume for contradiction that $M \setminus \varphi_U(N \times \mathbb{R})$ is nonempty.
Because the flow of $U$ preserves the leaves of $\mathcal{F}_\eta$, for any $x \in M \setminus \varphi_U(N \times \mathbb{R})$, the leaf $L'$ passing through $x$ is also contained in $M \setminus \varphi_U(N \times \mathbb{R})$.
By the same argument, $\varphi_U(N' \times \mathbb{R}) \subset M \setminus \varphi_U(N \times \mathbb{R})$, and this is an open subset of $M$.
This contradicts the fact that $M$ is connected.
Hence $M \setminus \varphi_U(N \times \mathbb{R})$ is empty.
\end{proof}

\begin{LL}[Candel - Conlon]
\label{1.13}
Let $\eta$ be a nonvanishing closed $1$-form on a smooth manifold $M$, and $U \in \mathfrak{X}(M)$ a complete vector field such that $\eta(U) = 1$.
For every leaf $N \in \mathcal{F}_\eta$, the image of the period morphism $\mathrm{Per}_{[\eta]}$ coincides with the set of all $t \in \mathbb{R}$ such that $\varphi_U^t(N) = N$.
\end{LL}

\begin{proof}
Let $t \in \mathbb{R}$ be such that $\varphi_U^t(N) = N$.
Since $N$ is connected, for every $x \in N$ there exists a path $\gamma_1$ joining $\varphi_U^t(x)$ to $x$ inside $N$, and we set 
$\gamma_2 : [0,t] \to M$, $s \mapsto \varphi_U^s(x)$.
The concatenation $\gamma_2 \cdot \gamma_1$ is then a piecewise smooth loop in $M$ such that
\begin{align*}
    \int_{\gamma_2 \cdot \gamma_1} \eta
        &= \int_{\gamma_2} \eta
         = \int_0^t \eta(U) \circ \varphi_U^s(x) \, ds
         = t.
\end{align*}
\smallskip

Conversely, every loop $\gamma : [0,1] \to M$ starting at a point $x \in N$ lifts uniquely to a loop 
$\tilde{\gamma}$ in $N \times \mathbb{R}$ starting at $(x,0) \in N \times \{0\}$ such that 
$\varphi_U|_{N \times \mathbb{R}} \circ \tilde{\gamma} = \gamma$.
Moreover, $\varphi_U|_{N \times \mathbb{R}}^* \eta = dt$.
Thus, if
\[
    t = \int_\gamma \eta = \int_{\tilde{\gamma}} dt,
\]
then $\tilde{\gamma}(1) \in N \times \{t\}$.  
Since $\gamma(1) \in N$ and the flow of $U$ preserves the leaves of $\mathcal{F}_\eta$, we conclude that 
$\varphi_U^t(N) = N$.
\end{proof}

\begin{TT}
\label{géné b-m}
Let $(M,[(\lambda,\eta)])$ be a exact LCS manifold of rank $1$ admitting an exact LCS pair $(\lambda,\eta)$ such that 
$\eta \wedge \lambda \wedge d\lambda^{\,n-1}$ is nonvanishing and the $\eta$-Liouville vector field $Z_\lambda$ is complete.  
Then $(M,[(\lambda,\eta)])$ is exactly conformally symplectomorphic to an LCS mapping torus.
\end{TT}

\begin{proof}
By Lemma~\ref{etaZnonzero}, the function $\eta(Z_\lambda)$ is nonvanishing.
Set $U := Z_\lambda / \eta(Z_\lambda)$. Since $Z_\lambda$ is complete by assumption, $U$ is complete and satisfies $\eta(U)=1$.
\smallskip

Fix a leaf $N \in \mathcal{F}_\eta$.
By Lemma~\ref{1.5}, the restriction 
$\varphi_U|_{N \times \mathbb{R}} : N \times \mathbb{R} \to M$ 
of the flow of $U$ to $N \times \mathbb{R}$ is a surjective local diffeomorphism.
Since the LCS rank of $(M,[(\lambda,\eta)])$ is $1$, by Lemma~\ref{1.13} we have 
$\{t \in \mathbb{R} \mid \varphi_U^t(N)=N\} \cong \mathbb{Z}$.
Let $k$ be a generator of $\mathrm{Im}(\mathrm{Per}_{[\eta]})$.  
For every $(x,t) \in N \times \mathbb{R}$ we then have
\begin{align*}
    \varphi_U|_{N \times \mathbb{R}} \circ \rho_{(\varphi_U^{-k}|_N,k)}(x,t)
        &= \varphi_U|_{N \times \mathbb{R}}(\varphi_U^{-k}(x),t+k) \\
        &= \varphi_U^{t+k}(\varphi_U^{-k}(x)) \\
        &= \varphi_U^{t}(x) \\
        &= \varphi_U|_{N \times \mathbb{R}}(x,t).
\end{align*}
Thus $\mathrm{Deck}(\varphi_U|_{N \times \mathbb{R}})$ is generated by 
$\rho_{(\varphi_U^{-k}|_N,k)}$.
By Lemma~\ref{1.13}, for every leaf $N' \in \mathcal{F}_\eta$ we have $\varphi_U^{k}(N') = N'$, and since $U$ and $Z_\lambda$ are colinear, we define a function 
$l \in C^\infty(M)$ by $\varphi_{Z_\lambda}^{l(x)}(x) = \varphi_U^{k}(x)$ for all $x \in M$.
If we replace $x$ by $\varphi_U^{-k}(y)$, we obtain that 
$\varphi_{Z_\lambda}^{-l\circ \varphi_U^{-k}(y)}(y) = \varphi_U^{-k}(y).$
Set $\psi := \varphi_U^{k}|_N$.
\smallskip

Since $\eta(U)=1$, we have $\varphi_U|_{N \times \mathbb{R}}^* \eta = dt$.
Thus $\overline{\lambda} := e^{-t} \varphi_U|_{N \times \mathbb{R}}{}^* \lambda$ 
is a Liouville form on $N \times \mathbb{R}$, with Liouville vector field 
$\overline{Z}$ defined by 
$\varphi_U|_{N \times \mathbb{R}}{}_* \overline{Z} = Z_\lambda$.
For every $(x,t) \in N \times \mathbb{R}$ and $s \in \mathbb{R}$ we have  
$\varphi_U|_{N \times \mathbb{R}} \circ \varphi_{\overline{Z}}^s(x,t)
 = \varphi_{Z_\lambda}^s \circ \varphi_U^t(x)$,
and we define
\[
    \rho : N \times \mathbb{R} \to N \times \mathbb{R} : (x,t) \mapsto 
    \varphi_{\overline{Z}}^{-l \circ \varphi_U^t(x)}(\psi(x),t).
\]
For each $(x,t) \in N \times \mathbb{R}$ we then compute
\begin{align*}
    \varphi_U|_{N \times \mathbb{R}} \circ \rho(x,t)
        &= \varphi_U|_{N \times \mathbb{R}} \circ 
           \varphi_{\overline{Z}}^{-l\circ \varphi_U^t(x)}(\psi(x),t) \\
        &= \varphi_{Z_\lambda}^{-l\circ \varphi_U^t(x)} 
           \circ \varphi_U^t(\psi(x)) \\
        &= \varphi_{Z_\lambda}^{-l\circ \varphi_U^{-k}(\varphi_U^t(\psi(x)))} 
           (\varphi_U^t(\psi(x))) \\
        &= \varphi_U^{-k}(\varphi_U^t(\psi(x))) \\
        &= \varphi_U^t(x)
        = \varphi_U|_{N \times \mathbb{R}}(x,t).
\end{align*}
Thus $\rho \in \mathrm{Deck}(\varphi_U|_{N\times \mathbb{R}})$.
If we fix $x \in N$ and choose a loop $\gamma$ corresponding to $\rho$ 
\footnote{i.e.\ $\gamma$ lifts to a path from $(x,0)$ to $\rho(x,0)$ in $N \times \mathbb{R}$.}
obtained by connecting $x$ to $\psi(x)$ inside $N$, and then following the flow of $U$ from $\psi(x)$ to $\varphi_U^{-k}(\psi(x)) = x$, we obtain
\[
    \int_\gamma \eta
        = \int_{0}^{-k} \eta(U) \circ \varphi_U^s(\psi(x)) \, ds
        = -k.
\]
This implies that $\rho$ coincides with $\rho_{(\psi,-k)}$, and  
$\rho^* \overline{\lambda} = e^{k} \overline{\lambda}$.
\smallskip

Identifying $N$ with $N \times \{0\}$ and $\lambda|_N$ with 
$\overline{\lambda}|_{N \times \{0\}}$, for any $x \in N$ and 
$A_x \in T_xN$ we compute
\begin{align*}
    e^k \lambda_x(A_x)
        &= (\rho^* \overline{\lambda})_{(x,0)}(A_{(x,0)}) \\
        &= \overline{\lambda}_{\varphi_{\overline{Z}}^{-l(x)}(\psi(x),0)}
           \big(\varphi_{\overline{Z}}^{-l(x)}{}_{*(\psi(x),0)}
           (\psi_* A_x)_{(x,0)}\big) \\
        &= e^{-l(x)} \overline{\lambda}_{(\psi(x),0)}
           ((\psi_* A_x)_{(x,0)}) \\
        &= e^{-l(x)} (\psi^* \lambda)_x(A_x).
\end{align*}
Thus $\psi \in \mathrm{Cont}(N,\lambda|_N)$ is a contactomorphism with conformal factor 
$h := l + k$.
\smallskip

Since the Liouville vector field $\overline{Z}$ is complete, its flow induces a diffeomorphism
\[
    \Theta : N \times \mathbb{R} \to N \times \mathbb{R}: (x,t) \mapsto \varphi_{\overline{Z}}^t(x,0),
\]
satisfying 
\[
    \Theta^* \overline{\lambda}_{(x,t)}
    = e^{t}\, \overline{\lambda}_{(x,0)}
    = e^{t} \lambda_x,
\]
for all $(x,t) \in N \times \mathbb{R}$.
Thus $\Theta$ defines an exact symplectomorphism between 
$(N \times \mathbb{R}, d(e^t \lambda|_N))$ and $(N \times \mathbb{R}, d\overline{\lambda})$.
By Lemma~\ref{relever de liouville}, the flow of $\overline{Z}$ commutes with the $\mathbb{Z}$-action generated by $\rho$, which implies that for all $(x,t) \in N \times \mathbb{R}$, we have
\begin{align*}
    \Theta^{-1} \circ \rho \circ \Theta(x,t)
        &= \Theta^{-1} \circ \rho \circ \varphi_{\overline{Z}}^t(x,0) \\
        &= \Theta^{-1} \circ \varphi_{\overline{Z}}^{t} \circ \rho(x,0) \\
        &= \Theta^{-1} \circ \varphi_{\overline{Z}}^{t-l(x)}(\psi(x),0) \\
        &= (\psi(x), t - l(x)) \\
        &= (\psi(x), t + k - h(x)).
\end{align*}
We thus recover the usual action generating the LCS mapping torus 
$N_{(\lambda|_N,\psi,k)}$, and by Proposition~\ref{2.16}, $\Theta$ induces on the quotient an exact conformal symplectomorphism between 
$N_{(\lambda|_N,\psi,k)}$ and $(M,[(\lambda,\eta)])$.
\end{proof}

\begin{RR}
    By Cartan’s formula,
\begin{align*}
    \mathcal{L}_{Z_\lambda} \lambda 
        &= \iota_{Z_\lambda} d\lambda + d\, \iota_{Z_\lambda} \lambda \\
        &= \iota_{Z_\lambda} d\lambda \\
        &= (1 + \eta(Z_\lambda))\, \lambda.
\end{align*}
This implies that
\[
    \frac{d}{dt}[\varphi_{Z_\lambda}^t{}^* \lambda]\big|_{t=s}
    = (1 + \eta(Z_\lambda)) \circ \varphi_{Z_\lambda}^s \ 
      \varphi_{Z_\lambda}^s{}^*\lambda, \ \  \text{with } \varphi_{Z_\lambda}^0{}^* \lambda = \lambda.
\]
This equation admits the unique solution
\[
    \varphi_{Z_\lambda}^t{}^* \lambda
    = \exp\!\Big( \int_0^t (1 + \eta(Z_\lambda)) 
        \circ \varphi_{Z_\lambda}^s \, ds \Big)\, \lambda.
\]

Moreover,
 \[\mathcal{L}_U \lambda = \frac{1}{\eta({Z_\lambda})} \mathcal{L}_{Z_\lambda} \lambda + d\big(\frac{1}{\eta({Z_\lambda})}\big) \wedge \iota_{Z_\lambda} \lambda = \frac{1}{\eta({Z_\lambda})} \mathcal{L}_{Z_\lambda} \lambda = \frac{1 + \eta({Z_\lambda})}{\eta({Z_\lambda})} \lambda.\]
Thus,
\[
    \varphi_U^t{}^* \lambda
    = \exp\!\Big(
        \int_0^t 
        \frac{1 + \eta(Z_\lambda)}{\eta(Z_\lambda)}
        \circ \varphi_U^s \, ds
      \Big)\, \lambda.
\]

Since $Z_\lambda$ is the $\eta$-Liouville vector field, we may then replace $t$ by $l$ to obtain
\begin{align*}
    \psi^* \lambda|_N
        &= \exp\!\Big(
            \int_0^l (1 + \eta(Z_\lambda)) 
            \circ \varphi_{Z_\lambda}^s \, ds
           \Big)\, \lambda|_N \\
        &= \exp\!\Big(
            \int_0^{k} 
            \frac{1 + \eta(Z_\lambda)}{\eta(Z_\lambda)}
            \circ \varphi_U^s \, ds
           \Big)\, \lambda|_N.
\end{align*}
This gives us a second expression for the conformal factor $h$ associated with $\psi$ in the previous theorem.
\end{RR}

\begin{DD}
A \textit{smooth discrete-time dynamical system} is the data of a smooth manifold $N$ and a diffeomorphism $\psi \in \mathrm{Diff}(N)$.
For every function $h \in C^\infty(N)$ and every $n \in \mathbb{N}_0$, we define the \textit{$n$-th Birkhoff partial sum} by
\[
    S_n(h) : N \to \mathbb{R} : 
    x \mapsto \sum_{i=0}^{n-1} h(\psi^i(x)),
\]
and the \textit{$n$-th Birkhoff partial average} by
\[
    A_n(h) : N \to \mathbb{R}: 
    x \mapsto \frac{S_n(h)(x)}{n}.
\]
\end{DD}

\begin{LL}
\label{3.8}
Let $(N,\xi)$ be a closed contact manifold, $\alpha$ a contact form associated with $\xi$, and $\psi$ a contactomorphism satisfying $\psi^* \alpha = e^h \alpha$, with $h \in C^\infty(N)$.
For every $k \in \mathcal{A}_{(\alpha,\psi)}$ and every $n \in \mathbb{N}_0$, there exists $l_n \geq n$ such that $k \notin \mathrm{Im}(A_{l_n}(h))$.
\end{LL}

\begin{proof}
For every $(x,t) \in N \times \mathbb{R}$ and every $n \in \mathbb{N}_0$, we have
\begin{align*}
    \rho_{(\psi,k-h)}^n(x,t)
        &= \big(\psi^n(x),\ t + n\big(k - \frac{1}{n} \sum_{i=0}^{n-1} h \circ \psi^i(x)\big)\big) \\
        &= \big(\psi^n(x),\ t + n\big(k - A_n(h)\big)\big).
\end{align*}
Assume for contradiction that the action generated by $\rho_{(\psi,k-h)}$ is properly discontinuous and that there exists $m \in \mathbb{N}_0$ such that for all $n \geq m$ we have $k \in \mathrm{Im}(A_n(h))$.
We may then construct a sequence $(x_n)_{n\in \mathbb{N}_0} \subset N$ such that for each $n \in \mathbb{N}_0$, we have $A_n(h)(x_n)=k$.
\smallskip

Since $N$ is compact, there exists a subsequence $((x_{\beta(n)},0))_{n\in \mathbb{N}_0} \subset ((x_n,0))_{n\in \mathbb{N}_0}$ and a subsequence 
\[
\big(\rho_{(\psi,k-h)}^{\beta'\circ \beta(n)}(x_{\beta'\circ \beta(n)},0)\big)_{n\in \mathbb{N}_0}
    \subset \big(\rho_{(\psi,k-h)}^{\beta(n)}(x_{\beta(n)},0)\big)_{n\in \mathbb{N}_0}
\]
which converge in $N \times \{0\}$ to $(y,0)$ and $(z,0)$ respectively.
Since every subsequence of a convergent sequence converges to the same limit, the subsequence 
$((x_{\beta' \circ \beta(n)},0))_{n\in \mathbb{N}_0}
    \subset ((x_{\beta(n)},0))_{n\in \mathbb{N}_0}$ 
also converges to $(y,0)$, which contradicts the fact that the action generated by $\rho_{(\psi,k-h)}$ is properly discontinuous.
\end{proof}

\begin{LL}
\label{3.5}
Let $N$ be a smooth manifold, $\alpha$ a contact form on $N$, and $\psi$ a contactomorphism satisfying $\psi^* \alpha = e^h \alpha$, with $h \in C^\infty(N)$.
For every constant $k \notin \mathrm{Im}(h)$, there exists a function $g \in C^\infty(N \times \mathbb{R})$ such that $\partial_t g + k$ is nonvanishing and 
\(
    g(\psi(x), t+1) = g(x,t) - h(x).
\)
\end{LL}

\begin{proof}
Assume that there exists $k \in \mathbb{R}$ such that $h < k$.
Choose $\epsilon \in \mathbb{R}$ such that $0 < \epsilon < k-h$ and $\epsilon/k < 1$.
Pick $\chi \in C^\infty(\mathbb{R})$ such that $\chi = 0$ on $]-\infty,0]$, 
$\chi = 1$ on $[1,+\infty[$, and $0 \leq \chi' < \frac{1}{1-\epsilon/k}$ on $\mathbb{R}$.
Set
\[
g(x,t)
    = \sum_{i=0}^\infty (1-\chi(t+1+i))\, h \circ \psi^{i}(x)
      - \sum_{i=0}^\infty \chi(t-i)\, h \circ \psi^{-i-1}(x),
\]
and one checks that
\begin{align*}
    g(\psi(x),t+1)
        &= \sum_{i=0}^\infty (1-\chi(t+2+i))\, h \circ \psi^{i+1}(x)
           - \sum_{i=0}^\infty \chi(t+1-i)\, h \circ \psi^{-i}(x) \\
        &= \sum_{i=1}^\infty (1-\chi(t+1+i))\, h \circ \psi^{i}(x)
           - \sum_{i=-1}^\infty \chi(t-i)\, h \circ \psi^{-i-1}(x) \\
        &= g(x,t) - (1-\chi(t+1))\, h(x) - \chi(t+1)\, h(x) \\
        &= g(x,t) - h(x).
\end{align*}

For every $s \in \mathbb{R}$, since $0 < k(1-\epsilon/k) - h$ and 
$0 \leq \chi'(s) < \frac{1}{1-\epsilon/k}$, we have
\[
    0 \leq \chi'(s)\, k(1-\epsilon/k) - \chi'(s)\, h 
      < k - \chi'(s)\, h.
\]
Since the supports of $\chi'(t+i+1)$ and $\chi'(t-i)$ are pairwise disjoint, we obtain
\[
0 < k
      - \sum_{i=0}^\infty \chi'(t+i+1)\, h \circ \psi^{i}(x)
      - \sum_{i=0}^\infty \chi'(t-i)\, h \circ \psi^{-i-1}(x)
    = \partial_t g + k.
\]
If there exists $k \in \mathbb{R}$ such that $k < h$,  
then applying the previous argument to 
$\psi^{-1}{}^* \alpha = e^{-h \circ \psi^{-1}} \alpha$, with 
$-h \circ \psi^{-1} < -k$, we construct a function 
$\Tilde{g} \in C^\infty(N \times \mathbb{R})$ satisfying 
$\Tilde{g}(\psi^{-1}(x),t+1) = \Tilde{g}(x,t) + h \circ \psi^{-1}(x)$ and 
$\partial_t \Tilde{g} - k > 0$.
Setting $g(x,t) := \Tilde{g}(x,-t)$, we then obtain 
$\partial_t g + k < 0$ and 
$g(\psi(x),t+1) = g(x,t) - h(x)$.
\end{proof}

\begin{LL}
\label{3.9}
Let $(N,\psi)$ be a smooth discrete-time dynamical system.
For every function $h \in C^\infty(N)$ and every $n \in \mathbb{N}_0$, there exists a function 
$f_n \in C^\infty(N)$ such that \(A_n(h) = h + f_n \circ \psi - f_n.\)
\end{LL}

\begin{proof}
Set \( f_n := \frac{1}{n} \sum_{i=1}^{n-1} S_i(h) \).  
We then check that
\begin{align*}
    h + f_n \circ \psi - f_n
        &= h + \frac{1}{n} \sum_{i=1}^{n-1} \sum_{j=0}^{i-1} h \circ \psi^{j+1}
           - \frac{1}{n} \sum_{i=1}^{n-1} \sum_{j=0}^{i-1} h \circ \psi^{j} \\
        &= h + \frac{1}{n} \sum_{i=1}^{n-1}
              \Big( \sum_{j=1}^{i} h \circ \psi^{j}
                   - \sum_{j=0}^{i-1} h \circ \psi^{j} \Big) \\
        &= h + \frac{1}{n} \sum_{i=1}^{n-1} \big( h \circ \psi^{i} - h \big) \\
        &= h + \frac{1}{n} \sum_{i=0}^{n-1} \big( h \circ \psi^{i} - h \big) \\
        &= \frac{1}{n} \sum_{i=0}^{n-1} h \circ \psi^{i}.
\end{align*}
\end{proof}

\begin{PP}
\label{3.10}
Let $N$ be a closed manifold, $\alpha$ a contact form on $N$, and $\psi$ a contactomorphism such that $\psi^* \alpha = e^h \alpha$, with $h \in C^\infty(N)$.
For every $k \in \mathcal{A}_{(\alpha,\psi)}$, there exists a rank~$1$ exact LCS pair $(\lambda,\eta)$ on $N_{(\alpha,\psi,k)}$ such that \(\eta \wedge \lambda \wedge d\lambda^{\,n-1}\) is nonvanishing.
\end{PP}

\begin{proof}
For every $k \in \mathcal{A}_{(\alpha,\psi)}$, by Lemmas~\ref{3.8} and \ref{3.9}, there exist 
$n \in \mathbb{N}_0$ and $f_n \in C^\infty(N)$ such that 
$k \notin \mathrm{Im}(A_n(h))$ and 
$\psi^*(e^{f_n}\alpha) = e^{A_n(h)} (e^{f_n}\alpha)$.
Thus, by Lemma~\ref{3.5}, one can construct a function 
$g \in C^\infty(N \times \mathbb{R})$ such that 
$\partial_t g + k$ is nonvanishing and \(g(\psi(x), t+1) = g(x,t) - A_n(h)(x)\) for all $(x,t) \in N \times \mathbb{R}$.
This allows us to construct a diffeomorphism
\[
    \overline{\sigma}_k : N \times \mathbb{R} \to N \times \mathbb{R} : (x,t) \mapsto (x, g(x,t) + tk + f_n(x)).
\]
Which satisfies, for every $(x,t) \in N \times \mathbb{R}$,
\begin{align*}
    \overline{\sigma}_k \circ \rho_{(\psi,1)}(x,t)
        &= \overline{\sigma}_k(\psi(x), t+1) \\
        &= (\psi(x), g(\psi(x),t+1) + (t+1)k + f_n \circ \psi(x)) \\
        &= (\psi(x), g(x,t) - A_n(h)(x) + f_n \circ \psi(x) + (t+1)k) \\
        &= (\psi(x), g(x,t) + f_n(x) - h(x) + (t+1)k) \\
        &= \rho_{(\psi,k-h)}(x, g(x,t) + tk + f_n(x)) \\
        &= \rho_{(\psi,k-h)} \circ \overline{\sigma}_k(x,t).
\end{align*}
Moreover, since $\rho_{(\psi,1)}^*(-t) = -t - 1$, if we set 
$\mu := -k\, t \circ \overline{\sigma}_k^{-1}$, then 
$\rho_{(\psi,k-h)}^* \mu = \mu - k$.
By Lemma~\ref{chapeaux pointu}, we then have that $\pi^* \lambda = e^{t+\mu} p_1^* \alpha$ and $\pi^* \eta = d\mu$ define a rank~$1$ exact LCS pair $(\lambda,\eta)$ on $N_{(\alpha,\psi,k)}$, where 
$\pi : N \times \mathbb{R} \to N_{(\alpha,\psi,k)}$ is the usual projection onto the quotient of $N \times \mathbb{R}$ by the $\mathbb{Z}$-action generated by $\rho_{(\psi,k-h)}$.
\smallskip

Since $\alpha$ is a contact form, 
$dt \wedge p_1^*(\alpha \wedge d\alpha^{n-1})$ is nonvanishing, and since 
$p_1 \circ \overline{\sigma}_k = p_1$, it follows that
\(
    \overline{\sigma}_k^{-1}{}^* dt \wedge p_1^*(\alpha \wedge d\alpha^{n-1})
    = d(t \circ \overline{\sigma}_k^{-1}) \wedge 
      p_1^*(\alpha \wedge d\alpha^{n-1})
\)
is also nonvanishing.
Moreover,
\[
    \pi^*(\eta \wedge \lambda \wedge d\lambda^{n-1})
        = -k\, e^{\,n(t - k t \circ \overline{\sigma}_k^{-1})}\,
          d(t \circ \overline{\sigma}_k^{-1}) \wedge 
          p_1^*(\alpha \wedge d\alpha^{n-1}).
\]
Thus $\eta \wedge \lambda \wedge d\lambda^{n-1}$ is indeed nonvanishing.
\end{proof}

\begin{TT}
\label{3.12}
A closed exact LCS manifold $(M,[(\lambda,\eta)])$ of rank~$1$ is exactly conformally symplectomorphic to the LCS mapping torus of a closed contact manifold if and only if there exists an exact LCS pair $(\lambda,\eta)$ generating the structure on $M$ such that $E(\lambda,\eta)^{c}$ is bounded.
\end{TT}

\begin{proof}
Since $M$ is closed, by Lemma~\ref{etaZnonzero} and Proposition~\ref{image élasic}, an exact LCS pair $(\lambda,\eta)$ satisfies 
$\eta \wedge \lambda \wedge d\lambda^{n}$ is nonvanishing if and only if 
$E(\lambda,\eta)^{c}$ is bounded.
Thus, by Theorem~\ref{géné b-m}, a closed exact LCS manifold $(M,[(\lambda,\eta)])$ of rank~$1$ admitting an exact LCS pair $(\lambda,\eta)$ such that $E(\lambda,\eta)^{c}$ is bounded is exactly conformally symplectomorphic to the LCS mapping torus of a closed contact manifold.
\smallskip

Conversely, if there exist a closed manifold $N$, a contact form $\alpha$, a contactomorphism 
$\psi \in \mathrm{Cont}(N,\alpha)$, and $k \in \mathbb{R}_0$ such that 
$(M,[(\lambda,\eta)])$ is exactly conformally symplectomorphic to the LCS mapping torus 
$N_{(\alpha,\psi,k)}$, then by Remark~\ref{rk 12} and Proposition~\ref{3.10}, there exists an exact LCS pair $(\lambda,\eta)$ generating the structure such that $\eta$ has rank~$1$ and $E(\lambda,\eta)^{c}$ is bounded.
\end{proof}
\subsection{Set of admissible values}

\begin{LL}
\label{rutabaga}
Let $N$ be a closed manifold, $\alpha$ a contact form on $N$, and $\psi$ a contactomorphism satisfying $\psi^*\alpha = e^h \alpha$, with $h \in C^\infty(N)$. 
For every $k \in \mathbb{R}$, if there exists $n \in \mathbb{N}_0$ such that $k \notin \mathrm{Im}(A_n(h))$, then $k \in \mathcal{A}_{(\alpha,\psi)}$.
\end{LL}

\begin{proof}
By Lemmas~\ref{3.5}, \ref{3.9}, and the proof of Proposition~\ref{3.10}, if there exists 
$n \in \mathbb{N}_0$ such that $k \notin \mathrm{Im}(A_n(h))$, one can construct a diffeomorphism 
$\overline{\sigma}_k \in \mathrm{Diff}(N \times \mathbb{R})$ which makes the $\mathbb{Z}$-actions generated by 
$\rho_{(\psi,1)}$ and $\rho_{(\psi,k-h)}$ commute.
Since the $\mathbb{Z}$-action generated by $\rho_{(\psi,1)}$ is properly discontinuous, this implies that 
$k \in \mathcal{A}_{(\alpha,\psi)}$.
\end{proof}

\begin{LL}
Let $N$ be a closed manifold, $\alpha$ a contact form on $N$, and $\psi$ a contactomorphism satisfying $\psi^*\alpha = e^h \alpha$, with $h \in C^\infty(N)$. 
The set of admissible values $\mathcal{A}_{(\alpha,\psi)}$ is an open subset of $\mathbb{R}_0$.
\end{LL}

\begin{proof}
From the proof of Proposition~\ref{3.10}, for every $k \in \mathcal{A}_{(\alpha,\psi)}$ there exist functions 
$f_k \in C^\infty(N)$ and $g_k \in C^\infty(N \times \mathbb{R})$ such that
\(\overline{\sigma}_k(x,t) = (x, g_k(x,t) + tk + f_k(x))\)
defines a diffeomorphism of $N \times \mathbb{R}$ satisfying 
\(\overline{\sigma}_k \circ \rho_\psi = \rho_{(\psi,k-h)} \circ \overline{\sigma}_k\).
If we set $\mu_k = -k t \circ \overline{\sigma}_k^{-1}$, then 
\(\rho_{(\psi,k-h)}^* \mu_k = \mu_k - k\), and we construct an exact LCS pair $(\lambda_k,\eta_k)$ on 
$N_{(\alpha,\psi,k)}$ defined by $\pi_k^* \lambda_k = e^{t + \mu_k} p_1^* \alpha$ and $\pi_k^* \eta_k = d\mu_k$, where $\pi_k : N \times \mathbb{R} \to N_{(\alpha,\psi,k)}$ is the usual quotient projection.
\smallskip

Since $\mu_k \circ \overline{\sigma}_k = -kt$, we have \begin{align*}
    \partial_t(\mu_k \circ \overline{\sigma}_k) &= (\partial_t \mu_k)\, (\partial_t (t \circ \overline{\sigma}_k)) \\
    &= (\partial_t \mu_k)(\partial_t g_k + k) \\
    &= -k.
\end{align*}
Hence,
\begin{align*}
    \frac{\eta_k(Z_{\lambda_k}) + 1}{\eta_k(Z_{\lambda_k})}
        = \frac{\partial_t \mu_k + 1}{\partial_t \mu_k}
        = \frac{\partial_t g_k}{-k}.
\end{align*}
Since $\lambda_k$ is nonvanishing, Proposition~\ref{image élasic} implies that 
$d_{c\eta_k}\lambda_k$ is nondegenerate if and only if $\partial_t g_k + ck$ is nonvanishing.
Moreover, if $\partial_t g_k + ck$ is nonvanishing, then \(\tilde{\sigma}_{ck}(x,t) = (x, g_k(x,t) + tck + f_k(x))\)
also defines a diffeomorphism of $N \times \mathbb{R}$ and satisfies
\(
    \tilde{\sigma}_{ck} \circ \rho_{(\psi,1)}
        = \rho_{(\psi,ck-h)} \circ \tilde{\sigma}_{ck}.
\)
\smallskip

The nondegeneracy of $d_{c\eta_k}\lambda_k$ therefore implies that the action generated by 
$\rho_{(\psi,ck-h)}$ is properly discontinuous.
Hence $\{ck \ | \ c \in E(\lambda_k,\eta_k)\} \subseteq \mathcal{A}_{(\alpha,\psi)}$ and $$ \mathcal{A}_{(\alpha,\psi)} = \bigcup_{k \in \mathcal{A}_{(\alpha,\psi)}} \{ck \ | \ c \in E(\lambda_k,\eta_k)\}.$$
Since $E(\lambda_k,\eta_k)$ is an open subset of $\mathbb{R}_0$, we conclude that 
$\mathcal{A}_{(\alpha,\psi)}$ is also open in $\mathbb{R}_0$.
\end{proof}

Let $N$ be a closed manifold, $\psi \in \mathrm{Diff}(N)$ a diffeomorphism, and $h \in C^\infty(N)$ a function.
For every $x \in N$, the sequence $(A_n(h)(x))_{n \in \mathbb{N}_0}$ is bounded in $[\min h, \max h]$ and therefore admits at least one accumulation value. We define
\[
    A_n^-(h)(x) := \inf_{i \geq n} A_i(h)(x)
    \ \ \text{and} \ \
    A_n^+(h)(x) := \sup_{i \geq n} A_i(h)(x).
\]
The sequences $(A_n^-(h)(x))_{n \in \mathbb{N}_0}$ and $(A_n^+(h)(x))_{n \in \mathbb{N}_0}$ are respectively increasing and decreasing, and satisfy for every $n \in \mathbb{N}_0$,
\[
    A_n^-(h)(x) \leq A_n(h)(x) \leq A_n^+(h)(x).
\]

\begin{PP}
\label{4.9}
Let $N$ be a closed manifold, $\alpha$ a contact form on $N$, and $\psi$ a contactomorphism satisfying $\psi^* \alpha = e^h \alpha$, with $h \in C^\infty(N)$.
The limits
\[
    \lim_{n \to +\infty} \inf_{x \in N} A^-_n(h)(x)
    \ \ \text{and} \ \
    \lim_{n \to +\infty} \sup_{x \in N} A^+_n(h)(x)
\]
do not depend on the choice of contact form.
\end{PP}

\begin{proof}
We show that for every $f \in C^\infty(N)$,
\[
    \lim_{n\to +\infty} \inf_{x \in N} A_n^-(h)(x)
    = 
    \lim_{n\to +\infty} \inf_{x \in N} A_n^-(h + f \circ \psi - f)(x).
\]
First observe that for all $x \in N$, $n \in \mathbb{N}_0$, and $f \in C^\infty(N)$,
\begin{align*}
    A_n(h + f \circ \psi - f)(x)
        &= A_n(h)(x)
           + \frac{1}{n} \sum_{i=0}^{n-1} (f \circ \psi^{i+1}(x) - f \circ \psi^i(x)) \\
        &= A_n(h)(x) + \frac{1}{n} (f \circ \psi^{n}(x) - f(x)).
\end{align*}
Thus,
\begin{align*}
    A_n^-(h + f \circ \psi - f)(x)
        &= \inf_{i \geq n} A_i(h + f \circ \psi - f)(x) \\
        &= \inf_{i \geq n} A_i(h)(x)
           + \frac{1}{i} (f \circ \psi^{i}(x) - f(x)) \\
        &\geq \inf_{i \geq n} A_i(h)(x)
           + \inf_{j \geq n} \frac{1}{j} (f \circ \psi^{j}(x) - f(x)) \\
        &= A_n^-(h)(x)
           + \inf_{j \geq n} \frac{1}{j} (f \circ \psi^{j}(x) - f(x)),
\end{align*}
and
\begin{align*}
    A_n^-(h)(x)
        &= \inf_{i \geq n} A_i(h)(x) \\
        &= \inf_{i \geq n} A_i(h + f \circ \psi - f)(x)
           - \frac{1}{i} (f \circ \psi^{i}(x) - f(x)) \\
        &\geq \inf_{i \geq n} A_i(h + f \circ \psi - f)(x)
           - \sup_{l \geq n} \frac{1}{l} (f \circ \psi^{l}(x) - f(x)) \\
        &= A_n^-(h + f \circ \psi - f)(x)
           - \sup_{l \geq n} \frac{1}{l} (f \circ \psi^{l}(x) - f(x)).
\end{align*}
This yields the chain of inequalities
\begin{align*}
    A_n^-(h + f \circ \psi - f)(x)
        &\geq A_n^-(h)(x)
            + \inf_{j \geq n} \frac{1}{j}(f \circ \psi^{j}(x) - f(x))
        \\
        &\geq A_n^-(h + f \circ \psi - f)(x)
            - \sup_{l \geq n} \frac{1}{l}(f \circ \psi^{l}(x) - f(x))
            + \inf_{j \geq n} \frac{1}{j}(f \circ \psi^{j}(x) - f(x)).
\end{align*}
Moreover, for all $n \in \mathbb{N}_0$,
\[
    \frac{\min f - \max f}{n}
    \leq \frac{f \circ \psi^n(x) - f(x)}{n}
    \leq \frac{\max f - \min f}{n}.
\]
Hence,
\[
    \lim_{n\to +\infty}
        \inf_{x \in N} \inf_{j \geq n}
        \frac{1}{j}(f \circ \psi^{j}(x) - f(x))
    = 0,
\]
and
\[
    \lim_{n\to +\infty}
        \sup_{x \in N} \sup_{l \geq n}
        \frac{1}{l}(f \circ \psi^{l}(x) - f(x))
    = 0.
\]
Therefore,
\[
    \lim_{n\to +\infty} \inf_{x \in N} A_n^-(h)(x)
    =
    \lim_{n\to +\infty} \inf_{x \in N} A_n^-(h + f \circ \psi - f)(x).
\]
The equality
\[
    \lim_{n\to +\infty} \sup_{x \in N} A_n^+(h)(x)
    =
    \lim_{n\to +\infty} \sup_{x \in N} A_n^+(h + f \circ \psi - f)(x)
\]
is proved in a similar way.
\end{proof}

\begin{TT}
\label{4.11}
Let $(N,\xi)$ be a closed contact manifold, $\alpha$ a contact form associated with $\xi$, and $\psi$ a contactomorphism satisfying $\psi^* \alpha = e^h \alpha$, with $h \in C^\infty(N)$.
Then
\[
    \mathcal{A}_{(\alpha,\psi)}
    =
    \ ]-\infty,\ \lim_{n\to +\infty} \inf_{x \in N} A_n^-(h)(x)[
    \ \cup\ 
    ]\lim_{n\to +\infty} \sup_{x \in N} A_n^+(h)(x),\ +\infty[\, .
\]
\end{TT}

\begin{proof}
Suppose for contradiction that \(k \notin [ \lim_{n\to + \infty} \inf_{x \in N} A^-_n(h)(x), \lim_{n\to + \infty} \sup_{x \in N} A^+_n(h)(x)]\) and that for every $n \in \mathbb{N}_0$, there exists $l_n \geq n$ such that $k \in \text{Im}(A_{l_n}(h))$.
We then construct a sequence $(x_n)_{n \in \mathbb{N}_0} \subset N$ such that $A_{l_n}(h)(x_n) = k$.
For every $n \in \mathbb{N}_0$, we then have
\[\inf_{x\in N} A_{l_n}^-(h)(x) \leq A_{l_n}^-(h)(x_n) \leq k \leq A_{l_n}^+(h)(x_n) \leq \sup_{x \in N} A_{l_n}^+(h)(x).\]
This contradicts the initial assumption. Hence there exists $m \in \mathbb{N}_0$ such that for all $n \geq m$, we have $k \notin \text{Im}(A_n(h))$.
Lemma \ref{rutabaga} then allows us to conclude that $k \in \mathcal{A}_{(\alpha,\psi)}$.
\smallskip

We now show that  
\[ \mathcal{A}_{(\alpha,\psi)} \subseteq \  ]-\infty, \lim_{n\to + \infty} \inf_{x \in N} A^-_n(h)(x)]  \cup  [\lim_{n\to + \infty} \sup_{x \in N} A^+_n(h)(x), + \infty[.\] 
Since $\mathcal{A}_{(\alpha,\psi)}$ is open, the desired equality follows.
\smallskip

If $k \in \mathcal{A}_{(\alpha,\psi)}$, then by Lemma \ref{3.8} there exists $m \in \mathbb{N}_0$ such that $k \notin \text{Im}(A_{m}(h)).$
By Lemma \ref{3.9}, there exists a function $f_m \in C^\infty(N)$ such that $A_m(h) = h+ f_m \circ \psi-f_m$, and in the proof of Proposition \ref{3.10} one constructs a function $\mu \in C^\infty(N \times \mathbb{R})$ satisfying $\rho_{(\psi,k-h)}^* \mu = \mu -k$.
Suppose for contradiction that 
\[\lim_{n\to + \infty} \inf_{x \in N} A^-_n(h)(x) < k < \lim_{n\to + \infty} \sup_{x \in N} A^+_n(h)(x).\]
Set 
\[K := N \times [-\max(\max h, k-\min h),-\min (\min h,k-\max h)].\]
We will show that for every sufficiently large $n \in \mathbb{N}_0$, there exist $l_n \geq n$ and $(y,s) \in K$ such that $\rho_{(\psi,k-h)}^{l_n}(y,s) \in K$.
But since $\rho_{(\psi,k-h)}^* \mu = \mu - k$, this will imply that $\mu(K)$ is unbounded, contradicting the compactness of $N$.
\smallskip

Choose $\delta > 0$ such that 
\[  \lim_{n\to + \infty} \inf_{x \in N} A^-_n(h)(x)< k- \delta < k+\delta < \lim_{n\to + \infty} \sup_{x \in N} A^+_n(h)(x).\]
Since the sequences $\{\inf_{x \in N}A_n^-(h)(x)\}_{n \in \mathbb{N}_0}$ and $\{\sup_{x \in N}A_n^+(h)(x)\}_{n \in \mathbb{N}_0}$ are respectively increasing and decreasing, for every $n \in \mathbb{N}_0$ we have 
\[ \inf_{x \in N} A^-_n(h)(x)< k- \delta < k+\delta < \sup_{x \in N} A^+_n(h)(x).\]
Let $n \in \mathbb{N}_0$ such that 
\[0 < \max \left(-{\min (\min h , k-\max h)},\max(\max h , k-\min h)\right) /{n} < \delta.\]
Set 
$$\epsilon := 1/2(\delta + {\min (\min h , k-\max h)}/{n}) > 0.$$
By definition of the infimum, there exists $x_- \in N$ such that 
$$A_n^-(h)(x_-) < \inf_{x \in N} A_n^-(h)(x) + \epsilon.$$
Similarly, since $A_n^-(h)(x_-) = \inf_{i \geq n} A_i(h)(x_-)$, there exists $n_- \geq n$ such that 
$$A_{n_-}(h)(x_-) < A_n^-(h)(x_-) + \epsilon.$$
Hence 
$$A_{n_-}(h)(x_-) < \inf_{x \in N} A^-_n(h)(x) + \delta + {\min(\min h , k-\max h)}/{n}.$$
Thus 
\[- {\min(\min h , k-\max h)}/{n} < k - A_{n_-}(h)(x_-).\]
This implies that 
\[\rho_{(\psi,k-h)}^{n_-}((x_-,0)) = (\psi^{n_-}(x_-),n_-(k-A_{n_-}(h)(x_-)))\in N \times \ ]-\min (\min h,k-\max h),+\infty[\]
In the same way, one finds $x_+ \in N$ and $n_+ \geq n$ such that 
\[k-A_{n_+}(h)(x_+) < -\max(\max h , k-\min h)/n.\]
This implies that 
\[\rho_{(\psi,k-h)}^{n_+}((x_+,0)) = (\psi^{n_+}(x_+),n_+(k-A_{n_+}(h)(x_+)))\in N \times \ ]-\infty, -\max \{\max h, k-\min h\}[.\]

Assume, without loss of generality, that $n_- \leq n_+$ and consider $\gamma$ a path from $(x_-,0)$ to $(x_+,0)$ in $N \times \{0\} \subset K$.
There are three possibilities for the position of $\rho^{n_+}_{(\psi,k-h)}(x_-,0)$. 
Either $\rho^{n_+}_{(\psi,k-h)}(x_-,0)$ lies to the left of $K$, in which case a point of $\gamma$ must pass through $K$. 
Or $\rho^{n_+}_{(\psi,k-h)}(x_-,0)$ lies inside $K$.
Or $\rho^{n_+}_{(\psi,k-h)}(x_-,0)$ lies to the right of $K$, in which case between the $n_-$-th and $n_+$-th iterate of $\rho_{(\psi,k-h)}$, the point $(x_-,0)$ must pass through $K$. 
We choose $l_n \in \{n_-, \ldots, n_+\}$ such that $\rho_{(\psi,k-h)}^{l_n}(x_-,0) \in K$. 
\end{proof}

\begin{TT}
\label{4.12}
Let $N$ be a closed manifold, $\alpha$ a contact form on $N$, and $\psi$ a contactomorphism such that $\psi^* \alpha = e^h \alpha$, with $h \in C^\infty(N)$.
We have 
\[ \inf_{f \in C^\infty(N)} \max_{x \in N}(h(x)+f\circ \psi(x)-f(x)) = \lim_{n \to \infty} \sup_{x \in N} A_n^+(h)(x) = \lim_{n \to \infty} \max_{x \in N} A_n(h)(x)\]
and 
\[ \sup_{f \in C^\infty(N)} \min_{x \in N}(h(x)+f\circ \psi(x)-f(x)) = \lim_{n \to \infty} \inf_{x \in N} A_n^-(h)(x) = \lim_{n \to \infty} \min_{x \in N} A_n(h)(x).\]
\end{TT}

\begin{proof}
For every $n\in \mathbb{N}_0$, by Lemma \ref{3.9} there exists a function $f_n \in C^\infty(N)$ such that $A_n(h) = h+f_n \circ \psi -f_n$.
Thus,
\[ \inf_{f \in C^\infty(N)}\max_{x \in N }(h(x)+f\circ \psi(x)-f(x)) \leq \max_{x \in N} A_n(h)(x) \leq \sup_{x\in N} A_n^+(h)(x)\footnote{Without additional assumptions, $A_n^+(h)$ is not continuous and hence not necessarily bounded.}.\]
Passing to the limit then yields
\[ \inf_{f \in C^\infty(N)}\max_{x \in N }(h(x)+f\circ \psi(x)-f(x)) \leq \lim_{n\to \infty}\max_{x \in N} A_n(h)(x) \leq \lim_{n\to \infty}\sup_{x\in N} A_n^+(h)(x).\]
By definition of the infimum, for every $\epsilon > 0$, there exists $f_\epsilon \in C^\infty(N)$ such that
\[\max_{x \in N}( h(x)+f_\epsilon\circ\psi(x)-f_\epsilon(x)) < \inf_{f \in C^\infty(N)}\max_{x \in N }(h(x)+f\circ \psi(x)-f(x)) +\epsilon/2.\]
Set $k = \epsilon/2 + \max_{x \in N}( h(x)+f_\epsilon\circ\psi(x)-f_\epsilon(x))$, so that $k >  h+f_\epsilon\circ\psi-f_\epsilon$.
By Lemmas \ref{castafior} and \ref{rutabaga}, we have $k\in \mathcal A_{(\alpha, \psi)}$, and by Theorem \ref{4.11}, we have $\lim_{n\to + \infty} \sup_{x \in N} A^+_n(h)(x) < k$.
Hence,
\[ \lim_{n\to + \infty} \sup_{x \in N} A^+_n(h)(x) < \inf_{f \in C^\infty(N)}\max_{x \in N }(h(x)+f\circ \psi(x)-f(x)) +\epsilon.\]
Letting $\epsilon$ tend to zero gives the desired equality.
The second equality is proved in a similar way.
\end{proof}

\newpage

\noindent{\bf{Acknowledgements.}} I am very grateful to Mélanie Bertelson for introducing me to LCS geometry, for her supervision and her help, without which this work would not have been possible.
``Pacôme Van Overschelde is a Research Fellow of the Fonds de la Recherche Scientifique - FNRS”

\printbibliography

@phdthesis{VanOverschelde2025,
  author       = {Pacôme Van Overschelde},
  title        = {Aspects quantitatifs en géométrie localement conformément symplectique.},
  school       = {Université libre de Bruxelles},
  year         = {2025},
  address      = {Bruxelles, Belgique}
}

@article{Lee1943,
  author       = {Lee, Hwa-Chung},
  title        = {A Kind of Even-Dimensional Differential Geometry and Its Application to Exterior Calculus},
  journal      = {American Journal of Mathematics},
  volume       = {65},
  number       = {3},
  pages        = {433--438},
  year         = {1943},
  doi          = {10.2307/2371967},
}

@article{Vaisman1985,
author = {Vaisman, Izu},
journal = {International Journal of Mathematics and Mathematical Sciences},
language = {eng},
pages = {521-536},
publisher = {Hindawi Publishing Corporation, New York},
title = {Locally conformal symplectic manifolds.},
volume = {8},
year = {1985},
}

@article{Lichnerowicz1983,
  author       = {Lichnerowicz, André},
  title        = {Sur les algèbres de Lie locales de Kirillov--Shiga},
  journal      = {Comptes Rendus de l'Académie des Sciences. Série I, Mathématique},
  volume       = {296},
  year         = {1983},
  pages        = {915--920},
}

@article{Libermann1955,
  author       = {Libermann, Paulette},
  title        = {Sur les structures presque complexes et autres structures infinitésimales régulières},
  journal      = {Bulletin de la Société Mathématique de France},
  volume       = {83},
  pages        = {195--224},
  year         = {1955},
  doi          = {10.24033/bsmf.1460},
}

@book{CandelConlon2000,
  author    = {Alberto Candel and Lawrence Conlon},
  title     = {Foliations I},
  series    = {Graduate Studies in Mathematics},
  volume    = {23},
  publisher = {American Mathematical Society},
  address   = {Providence, RI},
  year      = {2000},
  pages     = {402},
  isbn      = {978-0-8218-0809-2}
}

@article{Bazzoni_2018,
   title={On locally conformal symplectic manifolds of the first kind},
   volume={143},
   ISSN={0007-4497},
   DOI={10.1016/j.bulsci.2017.10.001},
   journal={Bulletin des Sciences Mathématiques},
   publisher={Elsevier BV},
   author={Bazzoni, Giovanni and Marrero, Juan Carlos},
   year={2018},
   month=mar, pages={1–57} }

@misc{chantraine2024productslocallyconformalsymplectic,
      title={Products of locally conformal symplectic manifolds}, 
      author={Baptiste Chantraine and Kevin Sackel},
      year={2024},
      eprint={2401.14918},
      archivePrefix={arXiv},
      primaryClass={math.SG},
}

@article{Belgun_2019,
   title={Locally conformally symplectic convexity},
   volume={135},
   ISSN={0393-0440},
   DOI={10.1016/j.geomphys.2018.10.001},
   journal={Journal of Geometry and Physics},
   publisher={Elsevier BV},
   author={Belgun, F. and Goertsches, O. and Petrecca, D.},
   year={2019},
   month=jan, pages={235–252} }

@article{7a23308de98840c3b920b67d847ed2f6,
title = "Some properties of locally conformal symplectic structures",
author = "Augustin Banyaga",
year = "2002",
doi = "10.1007/s00014-002-8345-z",
language = "English (US)",
volume = "77",
pages = "383--398",
journal = "Commentarii Mathematici Helvetici",
issn = "0010-2571",
publisher = "European Mathematical Society Publishing House",
number = "2",
}

@incollection{Banyaga2002,
  author       = {Banyaga, Augustin},
  title        = {On the Geometry of Locally Conformal Symplectic Manifolds},
  booktitle    = {Infinite Dimensional Lie Groups in Geometry and Representation Theory},
  pages        = {79--91},
  year         = {2002},
  publisher    = {World Scientific Publishing},
  address      = {River Edge, NJ},
  note         = {Conference held in Washington, DC, 2000},
}

\bigskip
\noindent\textsc{Université Libre de Bruxelles (ULB), Département de Géométrie Différentielle}\\
\noindent\textsc{CP218 Campus de la Plaine, 1050 Bruxelles, Belgium}\\
\noindent\textit{E-mail}: pacome.van.overschelde@ulb.be
\end{document}